\documentclass[a4paper,11pt,reqno]{amsart}

\usepackage{amssymb,amsmath,amsthm,graphicx,mathrsfs,listings,tikz,multicol}
\usepackage{enumerate}
\usepackage{float}
\usepackage{tikz-cd}
\usepackage{graphicx}
\usepackage{comment}

\newtheorem{thm}{Theorem}[section]
\newtheorem{lem}[thm]{Lemma}
\newtheorem{cor}[thm]{Corollary}
\newtheorem{prop}[thm]{Proposition}

\theoremstyle{definition}
\newtheorem{rem}[thm]{Remark}
\newtheorem{defi}[thm]{Definition}

\usetikzlibrary{arrows,fit}

\numberwithin{equation}{section}
\newcommand{\Pol}{{\rm Pol}}
\newcommand{\bG}{\mathbb{G}}

\newcommand{\R}{\mathbb{R}}

\newcommand{\E}{\mathbb{E}}
\newcommand{\C}{\mathbb{C}}
\newcommand{\N}{\mathbb{N}}
\newcommand{\Z}{\mathbb{Z}}

\newcommand{\T}{\mathbb{T}}
\newcommand{\GG}{\mathbb{G}}

\newcommand{\Span}{\mathrm{Span}}

\newcommand{\BMO}{{\rm BMO}}
\newcommand{\bmo}{\rm{bmo}}

\newcommand{\Tr}{\mathrm{Tr}}

\newcommand{\id}{\mathrm{id}}

\newcommand{\bi}{\begin{itemize}}
\newcommand{\ei}{\end{itemize}}
\newcommand{\be}{\begin{enumerate}}
\newcommand{\ee}{\end{enumerate}}
\newcommand{\bma}{\begin{bmatrix}}
\newcommand{\ema}{\end{bmatrix}}
\newcommand{\bpm}{\begin{pmatrix}}
\newcommand{\epm}{\end{pmatrix}}

\renewcommand{\it}{\item}

\newcommand{\nosep}{\setlength\itemsep{0em}}

\newcommand{\mc}{\mathcal}
\newcommand{\cM}{\mathcal{M}}

\newcommand{\sbinom}[2]{\left[\begin{smallmatrix} {#1} \\ {#2} \end{smallmatrix}\right]}

\tikzset{
  symbol/.style={
    draw=none,
    every to/.append style={
      edge node={node [sloped, allow upside down, auto=false]{$#1$}}}
  }
}

\definecolor{black}   {rgb}{0.0,0.0,0.0}
\definecolor{white}   {rgb}{1.0,1.0,1.0}
\definecolor{red}     {rgb}{1.0,0.0,0.0}
\definecolor{green}   {rgb}{0.0,0.5,0.0}
\definecolor{gray}    {rgb}{0.4,0.4,0.4}
\definecolor{blue}    {rgb}{0.0,0.0,1.0}
\definecolor{yellow}  {rgb}{1.0,1.0,0.0}
\definecolor{cyan}    {rgb}{0.0,1.0,1.0}
\definecolor{LightCyan}    {rgb}{0.5,1,1}
\definecolor{magenta} {rgb}{1.0,0.0,1.0}

\newcommand{\ot}{\otimes}

\newcommand{\la}{\langle}
\newcommand{\ra}{\rangle}
\newcommand{\vertiii}[1]{{\left\vert\kern-0.25ex\left\vert\kern-0.25ex\left\vert #1
    \right\vert\kern-0.25ex\right\vert\kern-0.25ex\right\vert}}

\renewcommand{\a}{\alpha}
\renewcommand{\b}{\beta}
\newcommand{\g}{\gamma}
\newcommand{\G}{\Gamma}
\renewcommand{\d}{\delta}
\newcommand{\D}{\Delta}
\newcommand{\e}{\varepsilon}
\newcommand{\vphi}{\varphi}
\newcommand{\z}{\zeta}

\renewcommand{\l}{\lambda}

\renewcommand{\k}{\kappa}

\newcommand{\s}{\sigma}

\renewcommand{\O}{\Omega}
\newcommand{\w}{\omega}

\title[BMO spaces of $\s$-finite von Neumann algebras and Fourier-Schur multipliers]{BMO spaces of $\s$-finite von Neumann algebras and Fourier-Schur multipliers on $SU_q(2)$}

\date{\today. MSC2020: 43A15, 46L67, 46L51. Keywords: compact quantum groups, BMO-spaces, Hardy spaces, Fourier and Schur multipliers. MC and GV are supported by the NWO Vidi grant ‘Noncommutative harmonic analysis and rigidity of operator algebras’, VI.Vidi.192.018.}

\author[Martijn Caspers]{Martijn Caspers}
\author[Gerrit Vos]{Gerrit Vos}
\address{TU Delft, EWI/DIAM, P.O.Box 5031, 2600 GA Delft,
	The Netherlands}
\email{m.p.t.caspers@tudelft.nl}
\email{g.m.vos@tudelft.nl}

\begin{document}

\begin{abstract}
We consider semi-group BMO spaces associated with an arbitrary $\s$-finite von Neumann algebra $(\mathcal{M}, \varphi)$. We prove that the associated row and column BMO spaces always admit a predual, extending results from the finite case. Consequently, we can prove that the semi-group BMO spaces considered are Banach spaces and they interpolate with $L_p$ as in the commutative situation, namely  $[\BMO(\mc{M}), L_p^\circ(\mc{M})]_{1/q} \approx L_{pq}^\circ(\mc{M})$. We then study a new class of examples. We introduce the notion of Fourier-Schur multiplier on a compact quantum group and show that such multipliers naturally exist for $SU_q(2)$.
\end{abstract}

\maketitle

\section{Introduction}\label{Sect=Intro}

Spaces of functions with Bounded Means of Oscillation (BMO spaces) play an eminent role in the theory of  harmonic analysis. They serve as so-called `end-point spaces' for many natural operators in harmonic analysis including singular integral operators and Fourier multipliers, see \cite{GrafakosBook}. More precisely, many singular integral operators and Fourier multipliers like the Riesz or Hilbert transform act boundedly as operators $L_p \rightarrow L_p, 1 < p < \infty$ and at the boundary   extend to bounded maps $L_\infty \rightarrow \BMO$. We call the latter bound an end-point estimate. Such endpoint estimates have several applications; one of the most important ones being that after interpolation they immediately yield $L_p$-boundedness with sharp constants.

For some singular integrals, like the Riesz and Hilbert transform, BMO spaces even provide optimal endpoint spaces. We mean this in the following sense (see \cite{FeffermanStein}, \cite{SteinBook}). Consider the Hardy-space $H^1$. By the celebrated Fefferman-Stein duality  we have $(H^1)^\ast \approx \BMO$. Then the Hilbert transform is bounded $H^1 \rightarrow L_1$. Moreover, the graph norm of the Hilbert transform as an unbounded map $L_1 \rightarrow L_1$ is equivalent to the $H^1$-norm (see \cite[Section 6.7.4]{GrafakosBook}). The same holds for the Riesz transform(s) if one takes all possible coordinates into account.

These and other results show that BMO and Hardy spaces occur naturally in the theory of singular integrals and their duality is of fundamental importance.

\vspace{0.3cm}

In the current paper we take a non-commutative viewpoint on BMO and Hardy spaces.
In this case the classical approach to BMO using cubes to measure the oscillation is replaced by an analysis of Markov semi-groups (in the commutative case diffusion semi-groups). In the commutative situation these ideas go back (at least) to \cite{Varopoulos}, \cite{StroockVaradhan}. Much more recently an analysis of duality and comparison of several such BMO-spaces was carried out in \cite{DuongYan1}, \cite{DuongYan2}.

The introduction of non-commutative semi-group BMO spaces was done by Mei \cite{Mei08} and further developed by Junge-Mei in \cite{JM12}. Their work is precedented by the theory of martingale BMO spaces \cite{PX97},  \cite{Pop00}, \cite{Mus03},  \cite{JM07} and \cite{JP14}. Most notably in the appendix of \cite{PX97} a duality  $(H_1)^* = \BMO$ is proven for a suitable notion of a Hardy space. Such martingale BMO spaces require the existence of a filtration of the von Neumann algebra. Many of the concrete cases of martingale BMO spaces concern  semi-classical von Neumann algebras (i.e. tensor products with a commutative von Neumann algebra) or a vector-valued situation where the filtration still comes from a commutative space. For some applications this structure is insufficient, see e.g. \cite{JMP14},  \cite{MeiLacunary}, \cite{Cas18}, \cite{CJSZ19} and one requires a true non-commutative version of BMO.

\vspace{0.3cm}

Here we shall take the approach to BMO from \cite{Mei08}, \cite{JM12} as a starting point. It assumes the existence of a Markov semi-group $\Phi = (\Phi_t)_{t\geq 0}$ on a finite (or semi-finite) von Neumann algebra $(\mc{M}, \tau)$, see Definition \ref{Dfn=Markov}.  \cite{JM12} considers various BMO-norms associated with this and its subordinated Poisson semigroup. We only consider the norm $\| \cdot \|_{\BMO_\Phi}$ (or $\| \cdot \|_{\BMO(\Phi)}$ in the notation of \cite{JM12}). For $x \in L_2(\mc{M})$ the column BMO-seminorm is then defined as
\begin{equation}\label{Eqn=BMOIntro}
    \Vert x \Vert^2_{\BMO^c_\Phi} =  \sup_{t \geq 0} \Vert  \Phi_t( \vert x - \Phi_t(x) \vert^2 )  \Vert_\infty,
\end{equation}
where the Markov maps $\Phi_t$ extend naturally to $L_2(\mc{M})$ and $L_1(\mc{M})$. Then column space $\BMO^c(\mc{M}, \Phi)$ is defined as the space of elements from $L_2(\mc{M})$ (minus some degenerate part) where the norm \eqref{Eqn=BMOIntro} is finite. Finally, $\BMO(\mc{M}, \Phi)$ is the intersection of $\BMO^c(\mc{M}, \Phi)$ and its adjoint row space.

\cite{JM12} establishes the natural interpolation results between BMO and $L_p$ by making use of Markov dilations and interpolation results for martingale BMO spaces. In the more general context of $\sigma$-finite von Neumann algebras a parallel study was carried out in \cite{Cas18} which again obtains such interpolation results through the Haagerup reduction method \cite{HJX10} and the finite case \cite{JM12}.   Both papers do this for several of the various BMO-norms defined in \cite{JM12}. The main advantage of considering the BMO-norm \eqref{Eqn=BMOIntro} as opposed to the norm $\| \cdot \|_{\bmo_\Phi}$ is that the Markov dilation is not required to have a.u. continuous path in order to apply complex interpolation.

There is a very subtle but important point that makes a difference between the current paper and \cite{Cas18}. In \cite{Cas18} BMO is defined by only considering $x$ in $\mc{M}$ and then taking an abstract completion with respect to the norm  \eqref{Eqn=BMOIntro} (or one of the other BMO-norms). This `smaller BMO space' has  the benefit that basic properties like the triangle inequality and completeness follow rather easily. Here we stay closer to the `larger BMO space' of $L_2$-elements with finite BMO-norm as defined above, and show that these basic properties still hold. We do this by proving a Fefferman-Stein duality result.

\vspace{0.3cm}

The contribution of this paper is twofold. Firstly, we study abstract BMO spaces of $\sigma$-finite von Neumann algebras. Instead of a direct $H^1$-BMO duality theorem, we will prove such a duality only for the column and row BMO spaces. This suffices for our purposes. The proof parallels the tracial proof in \cite{JMP14}. The main difficulty lies in the fact that $L_p$ spaces beyond tracial von Neumann algebras do not naturally intersect and we must deal with Tomita-Takesaki modular theory.

It should be mentioned that the $H^1$ Hardy spaces we construct here are abstract in nature and the question of whether every (column) BMO space has a natural Hardy space as its predual remains open. We refer to \cite{Mei08} and \cite[Open problems, p. 741]{JM12} for details about this question, where it was resolved under additional assumptions on the semi-group.

\begin{thm} \label{bmo-h1 duality thm-intro}
There exist Banach spaces $h_1^r(\cM, \Phi)$ and $h_1^c(\cM , \Phi)$ such that 
\[\BMO^c(\cM, \Phi) \cong h_1^r(\cM, \Phi)^*, \qquad \BMO^r(\cM, \Phi) \cong h_1^c(\cM, \Phi)^*.\]
\end{thm}

Within the construction of the preduals we need some $L_p$-module theory - see \cite{Pas73} and \cite{JS05}. In particular, we need to extend some results to the $\s$-finite case. We give an introduction to the theory and prove the necessary results in Section \ref{section Lp modules}.

The existence of these preduals for the column and row space then settles important basic properties of the BMO space itself, namely the triangle inequality and completeness of the normed space.

\begin{cor}  \label{Cor=IntroBanach}
$\BMO(\mc{M}, \Phi)$ is a Banach space.
\end{cor}

Finally, we show that the interpolation result of \cite{Cas18} still holds for our larger BMO space and extends \cite{JM12} beyond the tracial case. We refer to Appendix \ref{markov dilation section} and \cite{JM12}, \cite{Cas18} for the definition of a standard Markov dilation.

\begin{thm} \label{Thm=IntroInterpolation}
 If $\Phi$ is $\varphi$-modular and admits a  $\varphi$-modular standard Markov dilation, then for all $1 \leq p < \infty, 1 < q < \infty$,
\[
[\BMO(\mc{M}, \Phi), L_p^\circ(\mc{M})]_{1/q} \approx_{pq} L_{pq}^\circ(\mc{M}).
\]
Here $\approx_{pq}$ means that the Banach spaces are isomorphic and the norm of the isomorphism in both directions can be estimated by an absolute constant times $pq$.
\end{thm}

 We note that the modularity assumptions are only needed to carry out the Haagerup reduction method as in \cite{Cas18}. Many natural Markov semi-groups are modular or can be averaged to a modular Markov semi-group in case $\varphi$ is almost periodic, see \cite[Proposition 4.2]{CaspersSkalskiCMP}, \cite[Theorem 4.15]{OkayasuTomatsu}.

\vspace{0.3cm}

The second contribution we make consists of concrete examples for compact quantum groups. Theorem \ref{Thm=IntroInterpolation} as well as our construction of the preduals $h_1^r(\mc{M}, \Phi)$, $h_1^c(\mc{M}, \Phi)$ open the way for $L_p$-boundedness results on a wider range of multipliers. We give an application for multipliers on the quantum group $SU_q(2)$. In Section \ref{section SUq(2)}, we define Fourier-Schur multipliers on quantum groups which is an analogue of Fourier multipliers on group von Neumann algebras.

\begin{defi}
Let $\mathbb{G}$ be a compact quantum group and $T: \rm{Pol}(\mathbb{G}) \to \rm{Pol}(\mathbb{G})$ a linear map.  We call $T$ a  Fourier-Schur multiplier if the following condition holds.
 Let $u$ be any finite dimensional corepresentation on $\mc{H}$. Then there exists an orthogonal basis $e_i$ such that if  $u_{i,j}$ are the matrix coefficients with respect to this basis, then there exist numbers $c_{i,j} := c_{i,j}^u   \in \C$ such that
\[
    T u_{i,j} = c_{i,j} u_{i,j}.
\]
In this case $(c_{i,j}^u )_{i,j,u}$ is called the symbol of $T$.
\end{defi}

Basically, Fourier-Schur multipliers are Schur multipliers acting on the Fourier domain. We consider Fourier-Schur multipliers on $\GG_q := SU_q(2), q \in (-1, 1) \backslash \{ 0 \}$ associated with completely bounded Fourier multipliers on the torus $\T$.

The semigroups we use to define BMO are the Heat semi-group on $\T$ and the Markov semigroup $\Phi$ on $\GG_q$ constructed in Section \ref{SubSect=BMODef}. We use the shorthand notation $\BMO(\T)$, $\BMO(\GG_q)$ for the associated BMO spaces; see again Section \ref{SubSect=BMODef}.

\begin{thm}
Let $m \in \ell_\infty(\Z)$ with $m(0) = 0$ be such that the Fourier multiplier $T_m: L_\infty(\T) \to \BMO(\T)$ is completely bounded. Let $\tilde{T}_m: \rm{Pol}(\GG_q) \to \rm{Pol}(\GG_q)$ be the Fourier-Schur multiplier with symbol $(m(-i-j))_{i,j,l}$ with respect to the basis described in \eqref{basis vectors}. Then  $\tilde{T}_m$ extends to a bounded map
\[
    \tilde{T}_m^{(\infty)}: L_\infty(\GG_q) \to \BMO(\GG_q).
\]
Moreover $\Vert  \tilde{T}_m^{(p)}: L_\infty(\GG_q) \rightarrow \BMO(\GG_q) \Vert \leq  \Vert T_m : L_\infty(\T) \to \BMO(\T)\Vert_{cb}$.
\end{thm}

 Using the interpolation results of Section \ref{SubSect=Interpolation}, i.e. Theorem \ref{Thm=IntroInterpolation}, also the corresponding $L_p \rightarrow L_p$ follow. This is proved in Theorem \ref{Thm=FourierSchurLpVersion}.

In the proof we use our column and row $H^1$-$\BMO$ duality principle to show that Fourier-Schur multipliers extend from the weak-$\ast$ dense subalgebra of matrix coefficients of irreducible unitary corepresentations. The other important ingredient is a transference principle.

\vspace{0.3cm}

In the appendix, we give some comments on the operator space structures on $\BMO$. Also, we prove that the semigroup we use for the definition of our BMO space has a Markov dilation.

\vspace{0.3cm}

\noindent \textit{Structure of the paper.}
 In Section \ref{Sect=Preliminaries} we fix preliminary notation and introduce non-commutative $L_p$-spaces associated with $\s$-finite von Neumann algebras. Section \ref{section Lp modules} is devoted to $L_p$-module theory. We generalize some of the existing results from the tracial to the $\s$-finite case in order to apply them in the subsequent sections. Section \ref{section BMO} introduces BMO-spaces of $\s$-finite von Neumann algebras. We prove that the corresponding row and column spaces have preduals and gather corollaries. In other words, we prove Theorem \ref{bmo-h1 duality thm-intro} and Corollary \ref{Cor=IntroBanach}. In Section \ref{SubSect=Interpolation} we prove the interpolation result of Theorem \ref{Thm=IntroInterpolation}. The proof is  the same as in \cite{Cas18} provided that we can prove that an inclusion of a von Neumann algebra with expectation yields a 1-complemented BMO-subspace (this point was already surprisingly subtle in \cite{Cas18}). We give full details of this fact in Section \ref{SubSect=Interpolation}. In Section \ref{section SUq(2)} we turn to the examples. We introduce Fourier-Schur multipliers and show how to construct them on $SU_q(2)$. Finally, in the Appendix we gather results on operator space structures and Markov dilations.\\

We would like to express much gratitude towards the referees for their careful reading and numerous useful comments.

\section{Preliminaries}\label{Sect=Preliminaries}

\subsection{General notation} We use the convention $\N = \Z_{\geq 0}$.
 Following the convention in the literature for $L_p$-modules, inner products are linear in the second component and antilinear in the first. Dual actions are sometimes linear and sometimes antilinear (namely in the case of $L_p$-modules); whenever something is antilinear this will be explicitly mentioned. With an isomorphism (of Banach spaces), we shall mean a linear bijection that is bounded and whose inverse is also bounded. We write $\cong$ when the isomorphism is isometric. 

\subsection{Operator theory}


We use the following notation for tensor products:
\bi
\item $A \ot B$ for the algebraic tensor product of vector spaces.
\item $\mc{M} \bar{\ot} \mc{N}$ for the von Neumann algebraic tensor product.
\it $\mc{A} \ot_{\min} \mc{B}$ for the minimal tensor product of $C^*$-algebras.
\item $\mc{H} \ot_2 \mc{K}$ for the Hilbert space tensor product.
\ei

For general von Neumann algebra theory we refer to \cite{murphy} or Takesaki's books \cite{Takesaki1}, \cite{Takesaki2}, \cite{Takesaki3}. For the theory of operator spaces, see \cite{EffrosRuanBook} and \cite{PisierBook}.
 The following standard result shall be used several times in this paper. The proof follows directly from the definitions.

\begin{prop}[See \cite{ConwayBook}] \label{wk* cont extension}
Let $X, Y$ be Banach spaces and $T: X \to Y$ a bounded linear map. Then $T^*: Y^* \to X^*$ is weak-$*$/weak-$\ast$ continuous, i.e. normal.
\end{prop}

 Using this (and \cite[Chapter 1.22]{Sakai}) one proves that tensoring with the identity preserves normality. More precisely, for von Neumann algebras $\mc{M}, \mc{N}$ and a completely bounded normal operator $T: \mc{M} \to \mc{M}$, the map $1_{\mc{N}} \ot T$ extends to a normal operator on $\mc{N} \bar{\ot} \mc{M}$.

\vspace{0.3cm}

\noindent {\bf Convention:} All von Neumann algebras are assumed to be $\s$-finite. We will remind the reader of this convention a number of times in this paper.

\subsection{Compatible couples}\label{SubSect=compatible couple intro}

We will need some facts about compatible couples and compatible morphisms. We summarise some of the relevant theory from \cite{BerghLofstrom}.

\begin{defi}
A pair of Banach spaces $(A_0, A_1)$ are called a compatible couple if both are continuously embedded in some locally convex vector space $A$.
\end{defi}

We will mostly keep track of the continuous embeddings $i_0: A_0 \to A$ and $i_1: A_1 \to A$. One can define norms on the `intersection space' $i_0(A_0) \cap i_1(A_1)$ and `sum space' $i_0(A_0) + i_1(A_1)$ by
\begin{align*}
&\|a\|_{\cap} := \max\{\|i_0^{-1}(a)\|_{A_0}, \|i_1^{-1}(a)\|_{A_1}\}, \qquad a \in i_0(A_0) \cap i_1(A_1) \\
&\|a\|_{+} = \inf\{ \|a_0\|_{A_0} + \|a_1\|_{A_1}\ |\ i_0(a_0) + i_1(a_1) = a\}, \qquad  a \in i_0(A_0) + i_1(A_1).
\end{align*}
These norms turn the intersection and sum spaces into Banach spaces. When no confusion can occur, we will denote them simply by $A_0 + A_1$ and $A_0 \cap A_1$. \\

Let $(B_0, B_1)$ be another compatible couple given by embeddings $j_0: B_0 \to B$ and $j_1: B_1 \to B$. A pair of bounded maps $T_0: A_0 \to B_0$, $T_1: A_1 \to B_1$ are called compatible morphisms if they coincide on the (inverse image of the) intersection, i.e.
\[ j_0(T_0(a_0)) = j_1(T_1(a_1)), \ \ \ \ \text{ whenever } i_0(a_0) = i_1(a_1).
\]
If $(T_0, T_1)$ are compatible morphisms, then there exists a unique map $T: i_0(A_0) + i_1(A_1) \to j_0(B_0) + j_1(B_1)$ `extending' $T_0$ and $T_1$, i.e.
\begin{equation} \label{extension property}
    T(i_0(a)) = j_0(T_0(a)),\ T(i_1(b)) = j_1(T_1(b)), \ \ \ \ \ a \in A_0,\ b \in A_1.
\end{equation}

\subsection{$L_p$-spaces of $\s$-finite von Neumann algebras}
$L_p$-spaces corresponding to arbitrary  von Neumann algebras have been constructed by Haagerup \cite{Haa77} (see also \cite{Terp81}) and Connes-Hilsum \cite{Connes80}, \cite{Hil81} (see also Kosaki \cite{Kos84} in the $\sigma$-finite case). Here we will use the Connes-Hilsum definition. Each of the constructions can be recast in terms of the Haagerup definition; see for instance \cite[Section IV]{Terp81} for the isomorphism between Connes-Hilsum and Haagerup $L_p$-spaces.\\

Essential in the Connes-Hilsum construction is Connes' spatial derivative - see \cite{Connes80}, \cite{Terp81}. Let $\mc{M} \subseteq \mc{B}(\mc{H})$ be a  von Neumann algebra. Let $\psi$ be any fixed normal, semifinite faithful weight on the commutant $\mc{M}'$. For a normal, semifinite weight $\phi$ on $\mc{M}$, the spatial derivative is an (unbounded) positive (self-adjoint) operator denoted by
\[
D_\phi := d\phi / d\psi.
\]

\begin{rem}
The choice of $\psi$ will up to isomorphism not affect any of the constructions below. In particular it will  yield isometrically isomorphic non-commutative $L_p$-spaces.
We will assume henceforth that a choice for $\psi$ has been made implicitly and suppress it in the notation.
\end{rem}

\begin{rem}
In this paper we only deal with $\s$-finite von Neumann algebras $\mc{M}$: von Neumann algebras with a normal faithful state. In this case we may assume that $\mc{M}'$ is $\s$-finite as well, for example by considering the standard form of $\mc{M}$ \cite{Takesaki2}. This way we may assume that $\psi$ is a faithful normal state and we shall not require the general theory of weights on von Neumann algebras.
\end{rem}

The spatial derivative of a faithful normal state $\phi$ on $\mc{M}$ implements the modular automorphism group:
\begin{equation} \label{spatial derivative mag}
    \s^\phi_t(x) = D_\phi^{it} x D_\phi^{-it}, \ \ \ \ x \in \mc{M},\ t \in \R.
\end{equation}
We define the Tomita algebra
\[
    \mc{T}_\phi = \{x \in \mc{M} \mid t \mapsto \s_t^\phi(x) \text{ extends analytically to } \C\}.
\]
By \cite[Lemma VIII.2.3]{Takesaki2} $\mc{T}_\phi$ is a $\s$-weakly dense $*$-subalgebra of $\mc{M}$. Hence it is also $\s$-strong-* dense. \\

For $1 \leq p < \infty$ the space $L_p(\mc{M})$ is defined as the space of all closed densely defined operators $x$ on $\mc{H}$ such that $u \in \mc{M}$ for the $u$ from the polar decomposition $x = u|x|$, and $|x|^p = D_\phi$ for some $\phi \in \mc{M}_*^+$.  We define a trace on $L_1(\mc{M})$ as follows: let $x \in L_1(\mc{M})^+$ and $\phi \in \mc{M}_*^+$ be such that $x = D_\phi$. Then
\[
    \Tr(x) := \phi(1).
\]
The trace is then extended to $L_1(\mc{M})$ through the decomposition of an arbitrary operator into a linear combination of four positive operators. The norm on $L_p(\mc{M})$ is given by $\|x\|_p = \Tr(|x|^p)^{1/p}$. Further set  $L_\infty(\mc{M}) := \mc{M}$.
\\

Let $a, b \in L_p(\mc{M}), c \in L_q(\mc{M})$ with $1 \leq p,q \leq \infty$. Then $a + b$ and $ac$ are densely defined and preclosed. Their respective closures are called the strong sum and strong product and will simply be denoted by $a+b$ and $ac$. With these conventions  $a + b \in L_p(\mc{M})$  (turning  $L_p(\mc{M})$ into a Banach space) and $ac \in L_r(\mc{M})$ for $\frac1r := \frac1p + \frac1q$ with $r \geq 1$. Moreover, we have the  H\"older/Kosaki inequality:
\[
    \|a c\|_r \leq \|a\|_p \|c\|_q. \ \ \ \
\]  \\
In case $r = 1$ we have the trace property $\Tr(ac) = \Tr(ca)$ \cite[Proposition IV.13]{Terp81}.

%

\begin{rem} \label{p < 1 remark}
$L_p(\mc{M})$ may also be defined in the same way for $0 < p < 1$. It is not a normed space though. All we shall need in the current paper is that for $\frac12 \leq p < 1$ this space contains the product of two elements in $L_{2p}(\mc{M})$ and the square root of a positive element in $L_p(\mc{M})$ is in $L_{2p}(\mc{M})$.
\end{rem}

\noindent For $x \in L_p(\mc{M})$ we have $x^* \in L_p(\mc{M})$ with $\|x\|_p = \|x^*\|_p$  \cite[Prop IV.8]{Terp81}. In particular,
\begin{equation} \label{Lp squares} \|x^*x\|_{p/2}^{1/2} = \|x\|_p = \|x^*\|_p = \|xx^*\|_{p/2}^{1/2}. \end{equation}
There exists a duality pairing between $L_p(\mc{M})$ and $L_q(\mc{M})$ given by
\[
    \la x, y \ra = \Tr(x y), \ \ \ \ x \in L_q(\mc{M}), y \in L_p(\mc{M}),
\]
for $1 \leq p < \infty$ and $\frac1p + \frac1q = 1$. This induces an isometric isomorphism $L_p(\mc{M})^* \cong L_q(\mc{M})$. \\

\subsection{Compatible couples of $L_p$-spaces}
For finite von Neumann algebras, we have inclusions $L_q(\mc{M}) \subseteq L_p(\mc{M})$ for $\frac12 \leq p \leq q \leq \infty$. In the $\s$-finite case, the $L_p$-spaces are not included in each other as sets of operators on Hilbert spaces. However, they can be turned into a scale of compatible couples as follows.

Let $\mc{M}$ be a $\s$-finite von Neumann algebra. Fix a normal faithful state $\varphi$ on $\mc{M}$.  Fix $-1 \leq z \leq 1$.
For $x \in \mc{M},  \frac12 \leq p \leq \infty$ we have
\[
D_\vphi^{(\frac12-\frac{z}2)\frac1p  } x D_\vphi^{(\frac12+\frac{z}2)\frac1p} \in L_p(\mc{M}).
\]

For $\frac12 \leq p \leq q \leq \infty$ there are contractive embeddings
\[
    \k_{q, p}^{(z)}: L_q(\mc{M}) \to L_p(\mc{M}): D_\vphi^{(\frac12-\frac{z}2)\frac1q} x D_\vphi^{(\frac12+\frac{z}2)\frac1q} \mapsto D_\vphi^{(\frac12-\frac{z}2)\frac1p  } x D_\vphi^{(\frac12+\frac{z}2)\frac1p}, \qquad x \in \mc{M}.
\]
It is well-known that the images of these embeddings are dense for $1 \leq p \leq q$; this follows for instance from \cite[Theorem 9.1, Lemma 10.5]{Kos84} (this actually proves the result for the Haagerup construction, but as mentioned this can be recast in terms of the Connes-Hilsum construction).\\




Using the embeddings $\k_{p, 1}^{(z)}$ we may view $L_p(\mc{M})$ as a (dense) subspace of $L_1(\mc{M})$ and hence this turns all $L_p(\mc{M}), 1 \leq p \leq \infty$ simultaneously into a ($z$-dependent) scale of compatible couples.
For $x,y \in L_q(\mc{M})$ and $1 \leq p \leq  q \leq \infty$ we have
\begin{equation}\label{kappa identity}
\k^{(z)}_{q,p}(x)^* = \k^{(-z)}_{q,p}(x^*),    \qquad
 \k_{q,p}^{(-1)}(x) \k_{q,p}^{(1)}(y) = \k_{q/2, p/2}^{(0)}(xy).
\end{equation}

\noindent The embedding $\k_{\infty, 1}^{(z)}$ is `state-preserving' when we consider the trace on $L_1(\mc{M})$:
\begin{equation} \label{Eqn=TracePreserving}
    \Tr(\k^{(z)}_{\infty, 1}(x)) = \Tr(x D_{\vphi}) = \vphi(x), \ \ \ \ \ x \in \mc{M}.
\end{equation}
Indeed, for $x\in \mc{M}^+$ this follows from  \cite[Theorem III.14]{Terp81} and then use linearity  for general $x$.
 The following proposition is a special case of \cite[Theorem 5.1, Proposition 5.5]{HJX10}.

\begin{prop} \label{prop extension to Lp}
Let $T: \mc{M} \to \mc{M}$ be a unital completely positive (ucp) $\vphi$-preserving map such that $T \circ \sigma^\vphi_t = \sigma^\vphi_t \circ T, t \in \mathbb{R}$. Then $T$ extends to a positive contraction $T^{(p)}: L_p(\mc{M}) \to L_p(\mc{M})$ for $1 \leq p < \infty$ satisfying
\[
    T^{(p)}(\k_{\infty, p}^{(z)}(x)) = \k_{\infty, p}^{(z)}(T(x)), \ \ \ \ x \in \mc{M},
\]
which is independent of the choice of $-1 \leq z \leq 1$.  Additionally, $T^{(1)}$ is trace-preserving.
\end{prop}

\begin{proof}
We prove only the last statement. Consider first $x = x'D_\vphi \in L_1(\mc{M})$ for $x' \in \mc{M}$. With \eqref{Eqn=TracePreserving} we have
\[
    \Tr(T^{(1)}(x)) = \Tr(T(x') D_\vphi) = \vphi(T(x')) = \vphi(x') = \Tr(x).
\]
For general $x \in L_1(\mc{M})$ the statement follows by approximation.
\end{proof}

We recall that on the unit ball of $\mc{M}$ the strong topology coincides with the $\Vert \: \Vert_2$-topology generated by the GNS inner product $\langle x, y \rangle = \varphi(x^\ast y), x,y \in \mc{M}$. The following continuity property then follows from   \cite[Lemma 2.3]{JS05}.

\begin{prop} \label{continuity of embeddings}
Let $a_\l \in \mc{M}$ be a bounded net converging to $0$ in the strong topology. Then for any $1 \leq p < \infty$ and $x \in L_p(\mc{M})$:
\[
    \|a_\l x\|_p \to 0.
\]
\end{prop}

\section{$L_p$-module theory and duality results} \label{section Lp modules}

In this section we recall some $L_p$-module theory as introduced in \cite{JS05}. This theory builds upon the theory of Hilbert $C^*$-modules, see e.g. \cite{Pas73}, \cite{Lance95}.  It is also \cite{Pas73} that introduces the `GNS module' corresponding to a completely positive map. In the second part of this section, we extend some duality results to the $\s$-finite case; specifically, the duality relations of the $L_p$-module corresponding to the GNS modules. In Section \ref{section BMO}, we will use these results to construct a predual for $\BMO$ in the $\s$-finite case. \\

In the entire section $\mc{M}$ is a $\s$-finite von Neumann algebra with faithful normal state $\vphi$.

\subsection{General theory of $L_p$-modules}

\begin{defi} \label{Lp-modules def}
Let $1 \leq p \leq \infty$. A sesquilinear form $\la \cdot, \cdot \ra: X \times X \to L_{p/2}(\mc{M})$ on a right $\mc{M}$-module $X$ is called an {\em $L_{p/2}$-valued inner product} if it satisfies for $x, y \in X$ and $a \in \mc{M}$:
\be[(i)] \nosep
\it $\la x, x \ra \geq 0 $,
\it $\la x, x \ra = 0 \iff x = 0$,
\it $\la x, y \ra = \la y, x \ra^*$,
\it $\la x, ya \ra = \la x, y \ra a.$
\ee
A $L_{p/2}$-valued inner product defines a norm on $X$ given by 
\[ \|x\| := \|\la x, x\ra \|_{p/2}^{1/2}. \]
For $p< \infty$, $X$ is called an {\em $L_p$ $\mc{M}$-module} if it has a $L_{p/2}$-valued inner product and is complete with respect to the above norm. For $p = \infty$, we require that $X$ has a $L_\infty$-valued inner product and is complete in the topology generated by the seminorms
\[ x \mapsto \w(\la x, x \ra)^{1/2}, \ \ \ \w \in \mc{M}_*^+. \]
We call this the STOP topology (after \cite{JM12}).
\end{defi}

\begin{lem}{\cite[Proposition 3.2]{JS05}} \label{Lp cauchy schwarz}
For $x, y \in X$ there exists some $T \in \mc{M}$ with $\|T\| \leq 1$ such that $\la x, y \ra = \la x, x \ra^{\frac12} T \la y, y \ra^{\frac12}$. This implies the `$L_p$-module Cauchy Schwarz inequality':
\[ \|\la x, y \ra\|_{p/2} \leq \|x\| \|y\|. \]
\end{lem}

\begin{rem}
The norms defined here are a priori only quasinorms. However, Theorem \ref{principal Lp-modules} will show that they are in fact norms.
\end{rem}

An important class of $L_p$ $\mc{M}$-modules are the so-called {\em principal $L_p$-modules}. Recall the column space $L_p(\mc{M}; \ell_2^C(I))$ defined for $1 \leq p < \infty$ as the norm closure of finite sequences $x = (x_\a)_{\a \in I}$, $x_\a \in L_p(\mc{M})$, with respect to the norm
\[ \|x\|_{L_p(\mc{M}; \ell_2^C)} := \| (\sum_{\a \in I} |x_\a|^2)^{1/2}\|_p. \]
These spaces are isometrically isomorphic to $L_p(\mc{M} \bar{\ot} \mc{B}(\ell_2(I)))e_{1,1}$, the column subspace of $L_p(\mc{M} \bar{\ot} \mc{B}(\ell_2(I)))$, via
\[ (x_\a) \mapsto \bpm x_1 & 0 & \dots \\ x_2 & 0  & \dots \\ \vdots & \vdots \epm. \]
For $p = \infty$, we take the space of all sequences in $L_\infty(\mc{M})$ such that its image under the above map is in $L_\infty(\mc{M} \bar{\ot} \mc{B}(\ell_2(I)))$. See \cite{PX97} for more details about the above construction. \\

Now let $1 \leq p \leq \infty$ be fixed, $I$ be some index set and $(q_\a)_{\a \in I} \in \mc{M}$ be a set of projections. Consider the closed subspace
\[ X_p = \{ (x_\a)_{\a \in I}: x_\a \in q_\a L_p(\mc{M}), \sum_{\a \in I} x_\a^*x_\a \in L_{p/2}(\mc{M})\} \subseteq L_p(\mc{M}; \ell_2^C(I)). \]
We define an $L_{p/2}$-valued inner product on $X_p$ by
\[ \la x, y \ra = \sum_{\a \in I} (x_\a)^* y_\a. \]
We refer to \cite{JS05} for the fact that this is indeed a well-defined $L_{p/2}$-valued inner product. This makes $X_p$ into an $L_p$ $\mc{M}$-module. We call $X_p$ a \emph{principal $L_p$-module} and denote it by $\bigoplus_I q_\a L_p(\mc{M})$. \\

Note that we have the isometric isomorphism
\begin{equation} \label{X = QYP}
\bigoplus_I q_\a L_p(\mc{M}) \cong Q L_p(\mc{M} \bar{\ot} \mc{B}(\ell_2(I)))e_{1,1}, \ \ \ Q = \bpm q_1 & 0 & \dots \\ 0 & q_2 & \dots \\ \vdots & \vdots & \ddots \epm.
\end{equation}

This equation combined with the following general lemma (which has nothing to do with $L_p$-modules) will show that the family of principal $L_p$-modules $\bigoplus_I q_a L_p(\mc{M})$, $1 \leq p \leq \infty$, satisfies the expected duality relations (although the identifications become antilinear).

\begin{lem} \label{Lp module duality}
Let $\mc{N}$ be a $\s$-finite von Neumann algebra and let $P, Q \in \mc{N}$ projections. Then for $1 \leq p < \infty$, $\frac1p + \frac1{p'} = 1$ we have the following antilinear isometric isomorphism:
\[ (Q L_p(\mc{N})P)^* \cong Q L_{p'}(\mc{N})P. \]
\end{lem}

\begin{proof}
Let $1 \leq p < \infty$. Define $S_p := Q L_p(\mc{N})P \subseteq L_p(\mc{N})$. It follows (see for instance \cite[Theorem III.10.1]{ConwayBook}) that $S_p^* \cong L_{p'}(\mc{N})/S_p^\perp$, where $S_p^\perp = \{ b \in L_{p'}(\mc{N}): \Tr(S_p b) = 0\}$. Hence it suffices to prove $L_{p'}(\mc{N})/S_p^\perp \cong QL_{p'}(\mc{N})P$.\\

Let $a \in L_p(\mc{N})$, $b \in L_{p'}(\mc{N})$. Then $\Tr((QaP) b) = \Tr(a (PbQ))$, hence for $b \in L_{p'}(\mc{N})$:
\[ b \in S_p^\perp  \iff PbQ = 0 \iff Qb^*P = 0. \]
Therefore if we define the surjective map
\[ \Psi: L_{p'}(\mc{N}) \to QL_{p'}(\mc{N})P, \ \ \ \ \ \ b \mapsto Qb^*P, \]
then $\ker\Psi = S_p^\perp$ and hence the induced map $\Phi: L_{p'}(\mc{N})/S_p^\perp \to QL_{p'}(\mc{N})P$ is an isomorphism. $\Psi$ is contractive, hence $\Phi$ is also contractive. Conversely, for $b \in L_{p'}(\mc{N})$, we have
\[ P(b - PbQ)Q = PbQ - PbQ = 0, \]
hence $b - PbQ \in S_p^\perp$, or in other words $PbQ \in b + S_p^\perp$. Thus
\[ \|Qb^*P\| = \|PbQ\| \geq \|b + S_p^\perp\|. \]
This implies that $\Phi^{-1}$ is also contractive, so $\Phi$ is an isometric isomorphism.
\end{proof}

\begin{cor} \label{principal duality}
Let $(q_\a)_{\a \in I}$ be some family of projections. Then for $1 \leq p < \infty$, $\frac1p + \frac1{p'} = 1$, we have an antilinear isometric identification
\[ (\bigoplus_I q_\a L_p(\mc{M}))^* \cong \bigoplus_I q_\a L_{p'}(\mc{M}). \]
\end{cor}

The main theorem concerning $L_p$-modules states that every $L_p$-module is in fact isometrically isomorphic to a principal $L_p$-module.

\begin{thm}[Theorem 3.6 of \cite{JS05}] \label{principal Lp-modules}
Let $X$ be a right $L_p$ $\mc{M}$-module. Then there exists some index set $I$ and projections $(q_\a)_{\a \in I} \in \mc{M}$ such that
\[ X \cong \bigoplus_{\a \in I} q_\a L_p(\mc{M}). \]
\end{thm}

The following lemma allows us to transfer the duality results for principal $L_p$-modules to general families of $L_p$-modules satisfying certain requirements. The lemma is essentially copied from \cite[Corollary 1.13]{JP14} with some adjustments to go from the finite to the $\s$-finite case. It is in fact slightly more general to circumvent difficulties with finding an embedding $X_\infty \hookrightarrow X_p$.

\begin{lem} \label{principal family}
Let $(X_p)_{1 \leq p \leq \infty}$ be a family of $L_p$ $\mc{M}$-modules. Assume that there exist maps $I_{q,p}: X_q \to X_p$ ($q < \infty$) and $I_{\infty, p}: A \to X_p$ for some submodule $A \subseteq X_\infty$, that satisfy for $1 \leq p < r < q \leq \infty$:
\be[i)] \nosep
\it $I_{q,p}(xa) = I_{q,p}(x) \s_{i(\frac1p - \frac1q)}^\vphi(a)$ for $x \in X_q$ (or $x \in A$ if $q = \infty$), $a \in \mc{T}_\vphi$,
\it $I_{r,p} \circ I_{q,r} = I_{q,p}$,
\it $\k^{(0)}_{q/2, p/2}(\la x, y \ra_{X_q}) = \la I_{q,p}(x), I_{q,p}(y) \ra_{X_p}$ for $x, y \in X_q$ (or $x, y \in A$ if $q = \infty$),
\it $I_{\infty, p}(A)$ is dense in $X_p$.
\ee
Then there exists a family of projections $(q_\a)_{\a \in I} \in \mc{M}$ such that $X_p \cong \bigoplus_{\alpha \in I} q_\a L_p(\mc{M})$, $1 \leq p \leq \infty$.
\end{lem}

\begin{proof}
We give details only for those parts that differ from \cite[Corollary 1.13]{JP14}. One shows that the maps $I_{q,p}$ are automatically contractive embeddings. 
By applying Theorem \ref{principal Lp-modules} (which holds for $\s$-finite von Neumann algebras) to the $p = \infty$ case we acquire projections $(q_\a)$ such that $X_\infty \cong \bigoplus_{\a \in I}q_\a L_\infty(\mc{M})$, say through an isometric isomorphism of  $L_\infty$-modules $\vphi_\infty$. For $1 \leq p < \infty$, the embeddings $I_{\infty, p}$ allow us to `transfer' this map to $X_p$:
\[
    \vphi_p: I_{\infty, p}(A) \to \bigoplus_{\a \in I} q_\a L_p(\mc{M}), \ \ \ \vphi_p(I_{\infty,p}(x)) = \bigoplus_{\a \in I} \k^{(1)}_{\infty, p}(\vphi_\infty(x)_\a) = \bigoplus_{\a \in I} \vphi_\infty(x)_\a D_{\vphi}^{1/p}.
\]
\noindent We show that $\vphi_p$ preserves inner products; for $x, y \in A$:
\begin{align*}
    \la \vphi_p(I_{\infty,p}(x)), \vphi_p(I_{\infty,p}(y))\ra_{\bigoplus q_\a L_p} &= \sum_\a D_\vphi^{1/p}(\vphi_\infty(x)_\a)^* \vphi_\infty(x)_\a D_\vphi^{1/p}\\
    &= \k^{(0)}_{\infty, p/2}( \la \vphi_\infty(x), \vphi_\infty(y)\ra_{\bigoplus q_\a L_\infty} \\
    &= \k^{(0)}_{\infty, p/2}( \la x, y \ra_{X_\infty}) = \la I_{\infty,p}(x), I_{\infty,p}(y) \ra_{X_p}.
\end{align*}
Since $I_{\infty, p}(A)$ is dense in $X_p$, $\vphi_p$ extends to an isometric homomorphism on $X_p$. It turns out to be an isomorphism since we can use a similar argument to construct an inverse.
Next we show that $\vphi_p$ preserves the module structure (this was not an issue in the finite case); for $x \in A$, $a \in \mathcal{T}_\varphi$:
\begin{equation}\label{Eqn=Module}
\begin{split}
    \vphi_p(I_{\infty, p}(x)a) &= \vphi_p(I_{\infty, p}(x \s_{-\frac{i}p}^\vphi(a))) = \bigoplus_{\a \in I} \vphi_\infty(x \s_{-\frac{i}p}^\vphi(a))_\a D_\vphi^{1/p} \\
    &= \bigoplus_{\a \in I} \vphi_\infty(x)_\a \s_{-\frac{i}p}^\vphi(a) D_\vphi^{1/p} = \bigoplus_{\a \in I} \vphi_\infty(x)_\a D_\vphi^{1/p} a = \vphi_p(I_{\infty, p}(x)) a.
\end{split}
\end{equation}
Now let $a \in \mc{M}$ be arbitrary. By Kaplansky and strong density of $\mc{T}_\vphi$ in $\mc{M}$, we may choose a bounded net $(a_\lambda)_\lambda$ in $\mathcal{T}_\varphi$ converging to $a$ in the strong topology. Then by Proposition \ref{continuity of embeddings} we have
\[ \| I_{\infty, p}(x)(a- a_\lambda) \|_{X_p} = \|(a -  a_\l)^* \la I_{\infty, p}(x), I_{\infty, p}(x) \ra_{X_p} (a - a_\l)\|_{p/2}^{1/2} \rightarrow 0\]
and similarly $\Vert \vphi_p(I_{\infty, p}(x)) (a - a_\lambda) \Vert_{\bigoplus q_a L_p} \rightarrow 0$. Since $\varphi_p$ is continuous it follows that \eqref{Eqn=Module} holds for any $a \in \mc{M}$.

\end{proof}

\subsection{The GNS-module}

We now describe the GNS-module as introduced by \cite{Pas73}, but in the context of von Neumann algebras. Let $\Phi: \mc{M} \to \mc{M}$ be a completely positive map of von Neumann algebras. 
We define the $L_\infty$-valued inner product:
\[
    \la \sum_i a_i \ot b_i, \sum_j a_j' \ot b_j' \ra_\infty = \sum_{i,j} b_i^* \Phi(a_i^*a_j') b_j'
\]
and set $\mc{N}_0$ to be the quotient of $\mc{M} \ot \mc{M}$ by the set $\{ z \in \mc{M} \ot \mc{M}: \la z, z\ra = 0\}$.  \\

For $1 \leq p < \infty$, we define the $L_{p/2}$-valued inner product by simply taking the inclusion of $\mc{M}$ into $L_{p/2}(\mc{M})$ (see Remark \ref{p < 1 remark} for the case $1 \leq p < 2$):
\begin{equation} \label{Eqn=p-bracket}
    \la z, z' \ra_{p/2} = \k^{(0)}_{\infty, p/2} \left(\la z, z' \ra_\infty \right), \qquad z, z' \in \mc{M} \ot \mc{M}.
\end{equation}
This $L_{p/2}$-valued inner product gives rise to a norm $\|z\|_{p, \Phi} := \|\la z, z \ra_{p/2} \|_{p/2}^{1/2}$ on $\mc{N}_0$. We define $L_p(\mc{M} \ot_\Phi \mc{M})$ to be the Banach space completion of $\mc{N}_0$ with respect to this norm. \\

Next we define a module structure on $L_p(\mc{M} \ot_{\Phi} \mc{M})$. For $z \in \mc{M} \ot \mc{M}$ and $a \in \mc{T}_\vphi$, it is given by
\begin{equation} \label{module def}
    z \cdot a := z(1_\mc{M} \ot \s_{-\frac{i}{p}}(a)).
\end{equation}
Note that this module structure satisfies property (iv) of Definition \ref{Lp-modules def}. By Kaplansky and strong density of $\mc{T}_\vphi$ in $\mc{M}$, we can approach $a \in \mc{M}$ by a bounded net $(a_\l)_\l \in \mc{M}$ converging to $a$ in the strong topology. Setting $b_{\l, \mu} = a_\l - a_\mu$ and   using Proposition \ref{continuity of embeddings}, we have
\begin{align*}
    \|z \cdot b_{\l, \mu}\|_{p, \Phi} &= \| \la z \cdot b_{\l, \mu}, z \cdot b_{\l, \mu} \ra_{p/2}\|_{p/2}^{1/2} = \| b_{\l, \mu}^* \la z, z \ra_{p/2} b_{\l, \mu} \|_{p/2}^{1/2} \to 0.
\end{align*}

Hence we can extend \eqref{module def} for elements $a \in \mc{M}$, where the right hand side takes values in $L_p(\mc{M} \ot_{\Phi} \mc{M})$.  This right action is then strong/$\Vert \: \Vert_{p, \Phi}$-continuous on the unit ball of $\mc{M}$. \\

By the $L_p$-module Cauchy Schwarz inequality, the $L_{p/2}$-valued inner product and the module structure extend to the space $L_p(\mc{M} \ot_{\Phi} \mc{M})$. With this, $L_p(\mc{M} \ot_\Phi \mc{M})$ turns into a well-defined $L_p$ $\mc{M}$-module. \\


For $p = \infty$, we define $L_\infty(\mc{M} \ot_\Phi \mc{M})$ to be the completion with respect to the STOP topology, i.e. the one generated by the seminorms $z \mapsto \w(\la z, z \ra_\infty)^{1/2}$, $\w \in \mc{M}_*$. $\la \cdot, \cdot \ra_{\infty}$ is continuous in both variables on $\mc{M} \ot \mc{M}$ with respect to the STOP topology (and the weak-$\ast$ topology in the range); one can see this by writing $\la z, z' \ra_\infty = \la z, z \ra_\infty^{1/2} T \la z', z' \ra_\infty^{1/2}$ as in Lemma \ref{Lp cauchy schwarz} and, for $\w \in \mc{M}_*$, using the classical Cauchy Schwarz inequality on the bilinear form $(z, z') \mapsto \w(\la z, z' \ra_\infty)$. Hence $\la \cdot, \cdot \ra_{\infty}$ extends to an $\mc{M}$-valued inner product on $L_\infty(\mc{M} \ot_{\Phi} \mc{M})$. The module structure is simply given by $z \cdot a := z(1 \ot a)$.\\


\begin{prop} \label{GNS isom principal}
There exists a family of projections $(q_\a)_{\a \in I} \in \mc{M}$ such that $L_p(\mc{M} \ot_{\Phi} \mc{M}) \cong \bigoplus_I q_\a L_p(\mc{M})$, $1 \leq p \leq \infty$.
\end{prop}

\begin{proof}
To use Lemma \ref{principal family}, we must construct maps $I_{q,p}$ as in the assumptions of that lemma. The maps will be extensions of the identity map $\iota: \mc{M} \ot \mc{M} \to \mc{M} \ot \mc{M}$. For $q = \infty$, the space $A$ from the lemma will be $\mc{M} \ot \mc{M}$ and $I_{\infty, p}$ is simply the identity $\iota: A \to L_p(\mc{M} \ot_{\Phi} \mc{M})$. For $p \leq q < \infty$, the extensions exist because of the following estimate for $z \in \mc{M} \ot \mc{M}$:
\begin{align*}
 \|z\|_{q, \Phi} &= \|\la z, z\ra_{q/2}\|_{q/2}^{1/2} = \|\k^{(0)}_{\infty, q/2}(\la z, z \ra_\infty)\| _{q/2}^{1/2} \geq \| \k_{q/2, p/2}^{(0)} (\k^{(0)}_{\infty, q/2}(\la z, z \ra_\infty))\|_{p/2}^{1/2}  \\
&= \|\k^{(0)}_{\infty, p/2}(\la z, z \ra_\infty)\|_{p/2}^{1/2} = \|z\|_{p, \Phi}.
\end{align*}
It follows that $\iota$ extends to a contractive map $I_{q,p}: L_q(\mc{M} \ot_\Phi \mc{M}) \to L_p(\mc{M} \ot_\Phi \mc{M})$. The properties i)-iv) all follow from the previous constructions. 
Now we can apply Lemma \ref{principal family} to deduce the result.
\end{proof}

\begin{rem} \label{embedding of Linfty module}
We can deduce in hindsight the existence of the expected embedding
\[
    L_\infty(\mc{M} \ot_{\Phi} \mc{M}) \hookrightarrow L_p(\mc{M} \ot_{\Phi} \mc{M})
\]
through the identification with principal $L_p$-modules where the embedding is clear. We will need this observation later. In this case there is a common dense subset so there is no need to keep track of embeddings here; instead, we may `redefine' the GNS-modules for $1 < p \leq \infty$ to be closures within $L_1(\mc{M} \ot_{\Phi_t} \mc{M})$ instead of abstract completions, so that $L_q(\mc{M} \ot_{\Phi_t} \mc{M}) \subseteq L_p(\mc{M} \ot_{\Phi_t} \mc{M})$ for $1 \leq p \leq q \leq \infty$. Then through the identification with principal modules, we see that \eqref{Eqn=p-bracket} also holds for $z, z' \in L_\infty(\mc{M} \ot_{\Phi_t} \mc{M})$; this was not entirely trivial.
\end{rem}

Our next goal is to define duality results on the GNS-modules. To define a dual relation, we need to show that the bracket can be extended to a map taking arguments from different spaces. This follows easily through the identification with principal modules where this extension is evident. In the GNS-picture, the bracket is given by
\begin{equation}\label{dual GNS action}
    \la x, y\ra_{p,q} = D_\vphi^{1/p} \la x, y \ra_\infty D_\vphi^{1/q} = \k^{(z_{p,q})}_{\infty,r}(\la x, y \ra_{\infty})
\end{equation}
for $x, y \in \mc{M} \ot \mc{M}$ and $\frac1p + \frac1q = \frac1r$ with $1 \leq p, q, r \leq \infty$ but $p$ and $q$ not both $\infty$.\\

The (antilinear)  duality pairing is then defined as follows:
\begin{equation} \label{dual action}
    (x, y) = \Tr(\la x, y \ra_{p,q}), \ \ \ \ \ x \in L_p(\mc{M} \ot_{\Phi} \mc{M}),\ y \in L_q(\mc{M} \ot_{\Phi} \mc{M}),\ \frac1p + \frac1q = 1.
\end{equation}
This duality identifies $L_p(\mc{M} \ot_{\Phi} \mc{M})$ as a subspace of $L_q(\mc{M} \ot_{\Phi} \mc{M})^*$. Using the identification with principal modules, we can show that this inclusion is an (isometric) isomorphism.

\begin{cor} \label{duality GNS module}
For $1 \leq p < \infty$, $\frac1p + \frac1q = 1$, we have an antilinear isomorphism
\[ (L_p(\mc{M} \ot_\Phi \mc{M}))^* \cong L_q(\mc{M} \ot_\Phi \mc{M}). \]
\end{cor}

\begin{proof}
This follows from Proposition \ref{GNS isom principal} and Corollary \ref{principal duality}.
\end{proof}

\begin{rem}
The definition of $\la \cdot, \cdot \ra_{p,p}$ coincides with that of $\la \cdot, \cdot \ra_{p/2}$. Both notations make sense; the first refers to the inputs, the second to the output (and it corresponds to the term $L_{p/2}$-valued inner product). We will mostly be using the latter notation.
\end{rem}

\begin{rem} \label{equality of brackets}
Due to the tracial property, the embedding we choose to define the duality bracket does not matter. In particular, if $x \in L_1(\mc{M} \ot_{\Phi} \mc{M}) \cap L_2(\mc{M} \ot_{\Phi} \mc{M})$ and $y \in L_\infty(\mc{M} \ot_{\Phi} \mc{M}) \cap L_2(\mc{M} \ot_{\Phi} \mc{M})$ then
\[
    \Tr(\la x, y \ra_1) = \Tr(\la x, y \ra_{1, \infty})
\]
\end{rem}

\bigskip

In the next lemma we check that the inner product behaves as expected when we use, informally speaking, elements from $L_p(\mc{M})$ in the first tensor leg as inputs.  For this last lemma, we presume that $\Phi$ satisfies the conditions of Proposition \ref{prop extension to Lp} so that $\Phi^{(p/2)}$ exists.

\begin{lem}\label{Lp to Lp module}
Let $1 \leq p < \infty$, and let $\Phi$ be a unital completely positive (ucp) $\vphi$-preserving map such that $\Phi \circ \sigma^\vphi_t = \sigma^\vphi_t \circ \Phi$ for all $t \in \mathbb{R}$. The map
\[
    \Psi_p: \k_{\infty, p}^{(1)}(\mc{M}) \to L_p(\mc{M} \ot_{\Phi} \mc{M}), \ \ \ \ \k_{\infty, p}^{(1)}(x) \mapsto x \ot 1
\]
extends to a contractive mapping 
$\Psi_p: L_p(\mc{M}) \to  L_p(\mc{M} \ot_{\Phi} \mc{M})$. For $x, y \in L_p(\mc{M})$, $z = \sum_j a_j \ot b_j \in \mc{M} \ot \mc{M}$, it satisfies
\begin{align*}
    \la \Psi_p(x), \Psi_p(y) \ra_{p/2} &= \Phi^{(p/2)}(x^*y), \ \ \ \ \ \qquad 2 \leq p < \infty,\\
    \la \Psi_p(x), z \ra_{p/2} &= \sum_j \Phi^{(p)}(x^*a_j) b_j D_\vphi^{1/p}, \ \ \ \ \ 1 \leq p < \infty.
\end{align*}
\end{lem}

\begin{proof}
We first note the following identity for $x, y \in \mc{M}$:
\begin{equation} \label{ip embedding id}
    \la x \ot 1, y\ot 1 \ra_{p/2} = \k^{(0)}_{\infty, p/2}(\Phi(x^*y)) = \Phi^{(p/2)}(\k_{\infty, p/2}^{(0)}(x^*y)) \stackrel{\eqref{kappa identity}}= \Phi^{(p/2)}(\k_{\infty,p}^{(1)}(x)^*\k_{\infty,p}^{(1)}(y))
\end{equation}
Hence, by the generalised H\"older inequality
\begin{align*}
     \| x\ot 1\|_{p, \Phi} &= \|\Phi^{(p/2)}(\k_{\infty,p}^{(1)}(x)^*\k_{\infty,p}^{(1)}(x))\|_{p/2}^{1/2} \leq \|\k_{\infty,p}^{(1)}(x)^*\k_{\infty,p}^{(1)}(x)\|_{p/2}^{1/2} \\
     &\leq \|\k_{\infty,p}^{(1)}(x)^*\|_p^{1/2} \|\k_{\infty,p}^{(1)}(x)\|_p^{1/2} = \|\k_{\infty, p}^{(1)}(x)\|_p.
\end{align*}
This shows that $\Psi_p$ is contractive on $\k_{\infty, p}^{(1)}(\mc{M})$ and hence extends to a contractive mapping on $L_p(\mc{M})$.

Now let $x, y \in L_p(\mc{M})$ and take $(x_n), (y_n) \in \mc{M}$ such that $\k^{(1)}_{\infty, p}(x_n) \to_p x$ and $\k^{(1)}_{\infty, p}(y_n) \to_p y$. From Minkowski's inequality and the generalised H\"older inequality it follows that
\[ \k^{(1)}_{\infty, p}(x_n)^* \k^{(1)}_{\infty, p}(y_n) \to_{p/2} x^*y. \]
Hence by \eqref{ip embedding id} and continuity of $\Phi^{(p/2)}$:
\[ \la \Psi_p(x), \Psi_p(y) \ra_{p/2} = \lim_{n \to \infty} \la x_n \ot 1, y_n \ot 1 \ra_{p/2} = \lim_{n \to \infty} \Phi^{(p/2)}(\k_{\infty,p}^{(1)}(x_n)^*\k_{\infty,p}^{(1)}(y_n)) = \Phi^{(p/2)}(x^*y). \]
The final equality is proved with a very similar method and is left to the reader.
\end{proof}

\section{BMO spaces and BMO-$H_1$ duality} \label{section BMO}

In this section we construct $\BMO$ spaces of $\s$-finite von Neumann algebras and prove that they have a predual. We also prove the interpolation result of Theorem \ref{Thm=IntroInterpolation}.
$\mc{M}$ is again a $\s$-finite von Neumann algebra with faithful normal state $\vphi$.

\subsection{Introduction to Markov semigroups and BMO spaces} \label{SubSect=BMOintro}

\begin{defi}\label{Dfn=Markov}
A semigroup  $(\Phi_t)_{t \geq 0}$ of linear maps $\mc{M} \rightarrow \mc{M}$ is called a (GNS-symmetric) Markov semigroup if it satisfies the following conditions:
\be[i)] \nosep
\it $\Phi_t$ is normal ucp, $t \geq 0$,
\it $\vphi(\Phi_t(x)y) = \vphi(x\Phi_t(y))$, $x, y \in \mc{M}$, $t \geq 0$ (GNS-symmetry)
\it The mapping $t \mapsto \Phi_t(x)$ is strongly continuous, $x \in \mc{M}$.
\ee
The Markov semigroup is called $\vphi$-modular if $\Phi_t \circ \s_s^\vphi = \s_s^\vphi \circ \Phi_t$ for all $s \in \R$, $t \geq 0$.
\end{defi}

Note that by condition ii), $\vphi(\Phi_t(x)) = \vphi(x)$; in particular, the $\Phi_t$ are faithful.
If $\Phi := (\Phi_t)_{t \geq 0}$ is a $\vphi$-modular Markov semigroup, then by Proposition \ref{prop extension to Lp} there are extensions $\Phi_t^{(p)}: L_p(\mc{M}) \to L_p(\mc{M})$, where $\Phi_t^{(1)}$ is trace-preserving. Note that condition ii) implies, after appropriate approximations, that $\Phi_t^{(2)}$ is self-adjoint.



For the rest of this section we assume $\Phi = (\Phi_t)_{t \geq 0}$ to be a $\vphi$-modular Markov semigroup. We define closed subspaces of $\mc{M}$ and $L_p(\mc{M})$ as follows
\[
\begin{split}
 \mc{M}^\circ = & \{ x \in \mc{M}\ |\ \Phi_t(x) \to 0\ \s\text{-weakly as } t \to \infty\}, \\
 L_p^\circ(\mc{M}) = & \{ x \in L_p(\mc{M})\ |\ \|\Phi^{(p)}_t(x)\|_p \to 0,\ t \to \infty\}.
 \end{split}
 \]
Then \cite[Lemma 2.3]{Cas18} assures that the inclusions $\k^{(z)}_{q,p}$ restrict to contractive inclusions $L_q^{\circ}(\mc{M}) \to L_p^{\circ}(\mc{M})$ for $q \geq p$. \\

We record here two short lemmas for later use. We will need the generator $A_2$ of the semigroup $(\Phi_t^{(2)})_{t \geq 0}$, i.e. the positive self-adjoint unbounded operator such that $e^{-tA_2} = \Phi_t^{(2)}$; the existence is guaranteed by a very special case of the Hille-Yosida theorem and we refer to the papers \cite{Cip97} and \cite{GL95} for a more elaborate analysis of generators of Markovian semi-groups.

\begin{lem} \label{Lem=Limits of MS}
For each $x \in \mc{M}$, the net $\{\Phi_t(x)\}_{t \geq 0}$ converges $\s$-strongly as $t \to \infty$.  
\end{lem}

\begin{proof}
Let $x \in \mc{M}$ and write $x D_\vphi^{1/2} = \xi_1 + \xi_2$ for $\xi_1 \in \ker(A_2), \xi_2 \in \ker(A_2)^\perp$. Then 
\[
    \Phi_t(x) D_\vphi^{1/2} = \Phi_t^{(2)}(x D_\vphi^{1/2}) = e^{-tA_2} (\xi_1 + \xi_2) = \xi_1 + e^{-tA_2} \xi_2. 
\]
It follows by elementary spectral theory for unbounded operators that $e^{-tA_2}\xi_2 \to 0$ as $t \to \infty$. Therefore $\Phi_t(x) D_\vphi^{1/2}$ converges in the $L_2$-topology, i.e. $\Phi_t(x)$ is Cauchy within $\mc{M}$ in the $\| \cdot \|_2$-topology generated by the GNS inner product $\la x, y \ra = \vphi(x^*y), x, y \in \mc{M}$. Since the $\Phi_t$ are contractive, the net $\Phi_t(x)$ is bounded in $\mc{M}$. So as the $\| \cdot \|_2$-topology and the strong (and $\s$-strong) topology coincide on the unit ball, the net $\Phi_t(x)$ converges to an element in $\mc{M}$ in the strong (and $\s$-strong) topology.
\end{proof}

\begin{lem} \label{Lem=Circ duality}
Assume that $x \in L_1^\circ(\mc{M})$ is such that $\Tr(xz) = 0$ for all $z \in \mc{M}^\circ$. Then $x = 0$. 
\end{lem}

\begin{proof}
Let $y \in \mc{M}$ and set the $\s$-strong (hence $\s$-weak) limit $P(y) = \lim_{t \to \infty} \Phi_t(y)$, which exists by Lemma \ref{Lem=Limits of MS}. Then $y - P(y) \in \mc{M}^\circ$, hence we have
\[ \
    \Tr(xy) = \Tr(x(y - P(y))) + \Tr(xP(y)) = \Tr(xP(y)). 
\]
Now using condition ii) of Definition \ref{Dfn=Markov} and appropriate approximation, we can show that $\Tr(w \Phi_t(z)) = \Tr(\Phi_t^{(1)}(w) z)$ for $w \in L_1(\mc{M})$, $z \in \mc{M}$. Hence
\[
    \Tr(xP(y)) = \lim_{t \to \infty} \Tr(x \Phi_t(y)) = \lim_{t \to \infty} \Tr(\Phi_t^{(1)}(x) y) = 0
\]
since $x \in L_1^\circ(\mc{M})$. As $y \in \mc{M}$ was arbitrary, we must have $x = 0$. 
\end{proof}

For $x \in \mc{M}$ we define the column and row BMO-norm:
\[
 \|x\|_{\BMO^c_\Phi} = \sup_{t \geq 0} \|\Phi_t( |x - \Phi_t(x)|^2) \|_\infty^{1/2}; \ \ \ \ \|x\|_{\BMO^r_\Phi} = \|x^*\|_{\BMO^c_\Phi}.
\]
The BMO-norm is defined as $\|x\|_{\BMO_\Phi} = \max\{\|x\|_{\BMO^c_\Phi}, \|x\|_{\BMO^r_\Phi}\}$. This defines a seminorm by \cite[Proposition 2.1]{JM12}. \\

Since $\Phi$ is faithful, we see that for $x \in \mc{M}$, $\|x\|_{\BMO_\Phi} = 0$ implies that $x = \Phi_t(x)$ for all $t > 0$. This means that the above seminorms are actually norms on $\mc{M}^\circ$. \\

Next, we turn our attention to defining an analogous BMO-norm on the space $L_2(\mc{M})$ such as in \cite{JM12}. This turns out to be more involved in the $\s$-finite case. \\

The embedding $\k_{\infty, 1}^{(0)}$ allows us to define $\| \cdot \|_\infty$ on $L_1(\mc{M})$ (it takes values $\infty$ outside of $\k_{\infty, 1}^{(0)}(\mc{M})$). We will also denote this by $\| \cdot \|_\infty$. Then we can define analogous column and row BMO-(semi)norms on $L_2(\mc{M})$ by
\begin{equation}\label{Eqn=RowColumn}
\|x\|_{\BMO^c_\Phi} = \sup_{t \geq 0} \|\Phi_t^{(1)}(|x - \Phi_t^{(2)}(x)|^2)\|_\infty^{1/2}; \ \ \ \ \ \|x\|_{\BMO^r_\Phi} = \|x^*\|_{\BMO^c_\Phi}
 \end{equation}
We will only show later (at the end of this chapter) that these seminorms satisfy the triangle inequality. As with the corresponding norms on $\mc{M}$, these seminorms are norms on $L_2^\circ(\mc{M})$. Now we define the column BMO space as
\[ \BMO^c(\mc{M}, \Phi) = \{x \in L_2^\circ(\mc{M})\ |\ \|x\|_{\BMO^c_\Phi} < \infty\} \]
and we define the row BMO space as the adjoint of the column BMO space with norm as in \eqref{Eqn=RowColumn}. We  emphasize  that we have thus constructed a column (resp. row) BMO-norm both on $\mc{M}^\circ$ and $L_2^\circ(\mc{M})$ which by mild abuse of notation are  denoted in the same way. They are identified by the right embedding for the column norm and the left embedding for the row norm:
\begin{equation} \label{right bmo embedding}
\begin{split}
    \|\k_{\infty, 2}^{(1)}(x)\|_{\BMO^c_\Phi} = & \|x D_\vphi^{1/2} \|_{\BMO^c_\Phi} = \|x\|_{\BMO^c_\Phi}, \\
    \|\k_{\infty, 2}^{(-1)}(x)\|_{\BMO^r_\Phi} = & \| D_\vphi^{1/2} x \|_{\BMO^r_\Phi} = \|x\|_{\BMO^r_\Phi},
\end{split}
\end{equation}
where $x \in \mc{M}^\circ$. These equalities are straightforward to check. Since clearly $\|x\|_{\BMO^c_\Phi} \leq 4\|x\|_{\infty}^2$ for $x \in \mc{M}^\circ$, it follows that $\k_{\infty,2}^{(1)}$ embeds $\mc{M}^\circ$ into $\BMO^c(\mc{M}, \Phi)$, and similarly $\k_{\infty,2}^{(-1)}$ embeds $\mc{M}^\circ$ into $\BMO^r(\mc{M}, \Phi)$. \\

The first thought for a definition of the BMO-norm would be $\max\{ \|x\|_{\BMO^c_\Phi}, \|x\|_{\BMO^r_\Phi}\}$, similarly to the definition on $\mc{M}$. However, this is not a suitable definition for the following reason.
The equalities \eqref{right bmo embedding} show  how the right and left embeddings of $\mc{M}$ in $L_2(\mc{M})$ preserve the column and row norms respectively. However, there is no embedding of $\mc{M}$ into $L_2(\mc{M})$ that would preserve the maximum of these norms.

Instead, we embed $\BMO^c(\mc{M}, \Phi)$ and $\BMO^r(\mc{M}, \Phi)$ in $L_1^\circ(\mc{M})$ through the embeddings $\k^{(-1)}_{2, 1}$ and $\k^{(1)}_{2, 1}$ respectively. This turns $(\BMO^c(\mc{M}, \Phi),\ \BMO^r(\mc{M}, \Phi))$ into a compatible couple.   The following diagram commutes:

\begin{center}
\begin{tikzcd}[row sep=small]
    && L_2^\circ(\mc{M}) \arrow{rrdd}{\k^{(-1)}_{2, 1}} && \\
    && \BMO^c(\mc{M}, \Phi) \arrow{rrd} \arrow[u, symbol=\subseteq]&& \\
    \mc{M}^\circ \arrow{rruu}{\k^{(1)}_{\infty, 2}} \arrow{rru} \arrow{rrrr}{\k^{(0)}_{\infty, 1}} \arrow{rrd} \arrow{rrdd}[swap]{\k^{(-1)}_{\infty, 2}} &&&& L_1^\circ(\mc{M}) \\
    && \BMO^r(\mc{M}, \Phi) \arrow{rru} \arrow[d, symbol=\subseteq] && \\
    && L_2^\circ(\mc{M}) \arrow{rruu}[swap]{\k^{(1)}_{2, 1}} && \\
\end{tikzcd}
\end{center}
We define
\[
    \BMO(\mc{M}, \Phi) = \k^{(-1)}_{2, 1}(\BMO^c(\mc{M}, \Phi)) \cap \k^{(1)}_{2, 1}(\BMO^r(\mc{M}, \Phi))
\]
to be the intersection space, and for $x \in \BMO(\mc{M}, \Phi)$ we denote by
\[
x_c \in \BMO^c(\mc{M}, \Phi),  \qquad x_r \in \BMO^r(\mc{M}, \Phi)
\]
 the elements such that $\k_{2,1}^{(-1)}(x_c) = x = \k_{2, 1}^{(1)}(x_r)$. The norm on $\BMO(\mc{M}, \Phi)$ is defined as
\[
    \|x\|_{\BMO_\Phi} = \max\{ \|x_c\|_{\BMO^c_\Phi}, \| x_r \|_{\BMO^r_\Phi}\}.
\]
When no confusion can occur, we omit the reference to the semigroup in the notation of the various BMO-norms and just write, for instance, $\| \cdot \|_{\BMO}$.

We check that $\k^{(0)}_{\infty, 1}$ is indeed an embedding of $\mc{M}^\circ$ into $\BMO(\mc{M})$ that preserves $\| \cdot \|_{\BMO}$:
\[
\begin{split}
    \|\k^{(0)}_{\infty,1 }(z)\|_{\BMO} = & \max\{ \|\k^{(1)}_{\infty, 2}(z)\|_{\BMO^c}, \|\k^{(-1)}_{\infty, 2}(z)\|_{\BMO^r} \} \\
    = &\max\{\|z\|_{\BMO^c}, \|z\|_{\BMO^r}\} = \|z\|_{\BMO}.
\end{split}
\]
The next estimate shows that $L_1^\circ(\mc{M})$ contains the closure of $\k_{\infty, 1}^{(0)}(\mc{M}^\circ)$ with respect to $\| \cdot \|_{\BMO}$, as  expected. \\

\begin{lem} \label{2-norm <= BMO-norm}
For $x \in L_2^\circ(\mc{M})$, we have $\|x\|_2 \leq \|x\|_{\BMO^c}$ and $\|x\|_2 \leq \|x\|_{\BMO^r}$. Hence for $x \in \BMO(\mc{M}, \Phi)$, we have
\[
    \|x\|_{\BMO} \geq \max\{\|x_c\|_2, \|x_r\|_2\} \geq \|x\|_1.
\]
\end{lem}

\begin{proof}
Let $x \in L_2^\circ(\mc{M})$. If $\|x\|_{\BMO^c} = \infty$ then the inequality trivially holds. Now assume that $\|x\|_{\BMO^c} < \infty$. Then for all $t \geq 0$ there exists a $y_t \in \mc{M}$ such that $\Phi_t^{(1)}|x - \Phi_t^{(2)}(x)|^2 = \k_{\infty, 1}^{(0)}(y_t)$.  \\

Let $\e > 0$. Then we can find $t > 0$ such that $\|\Phi_t^{(2)}(x)\|_2 < \e$. Then since $\Phi_t^{(1)}$ is trace-preserving:
\begin{align*}
    \|x\|_2 &\leq \|x - \Phi_t^{(2)}(x)\|_2 + \e = \Tr(|x - \Phi_t^{(2)}(x)|^2)^{1/2} + \e = \Tr(\Phi_t^{(1)}|x - \Phi_t(x)|^2)^{1/2} + \e \\
    &= \Tr(\k^{(0)}_{\infty, 1}(y_t))^{1/2} + \e = \vphi(y_t)^{1/2} + \e \leq \|y_t\|_\infty^{1/2} + \e \leq \|x\|_{\BMO^c} + \e.
\end{align*}
Since $\|x\|_2 = \|x^*\|_2$, we also get $\|x\|_2 \leq \|x\|_{\BMO^r}$. The final statement follows from the definition of $\| \cdot \|_{\BMO}$ and contractivity of $\k^{(z)}_{2, 1}$. This finishes the proof.
\end{proof}

It is not a priori clear whether $\BMO(\mc{M}, \Phi)$ is complete. However, this will follow as a corollary from the result of the next subsection, which provides an `artificial' predual to $\BMO(\mc{M}, \Phi)$. \\

\subsection{A predual of BMO} \label{bmo-h1 duality}

We dedicate this section to proving the following theorem:
\begin{thm} \label{bmo-h1 duality thm}
There exist Banach spaces $h_1^r(\cM, \Phi)$ and $h_1^c(\cM , \Phi)$ such that 
\[\BMO^c(\cM, \Phi) \cong h_1^r(\cM, \Phi)^*, \qquad \BMO^r(\cM, \Phi) \cong h_1^c(\cM, \Phi)^*.\]
\end{thm}

In this part we will suppress the reference to $\mc{M}$ and $\Phi$ in the notation of $\BMO^c, \BMO^r$ and their preduals $h_1^r, h_1^c$.\\

In the finite case a predual for BMO was found in \cite[Section 5.2.3]{JM12}, see also \cite[Appendix A]{JMP14}. Our proof mostly follows the lines of \cite{JMP14}, although we will not attempt to define a sum space $h_1$. Also, our predual of $\BMO^c$ will instead be $h_1^r$ and vice versa, which makes the identification in Theorem \ref{bmo-h1 duality thm} linear instead of antilinear.

\begin{proof}[Proof of Theorem 4.5]
Since $\BMO^r$ lies within $L_2^\circ(\mc{M})$, we have at our disposal an inner product that can provide us with a duality bracket. We take the Hahn-Banach norm relation as the definition of the norm of $h_1^c$:
\[ \|y\|_{h_1^c} = \sup_{\|x\|_{\BMO^r} \leq 1} |\Tr(x y)|, \ \ \ \ \ y \in L_2^\circ(\mc{M}). \]
which would be a well-defined norm even if $\| \cdot \|_{\BMO^r}$  wouldn't satisfy the triangle inequality. To see that $\|y\|_{h_1^c} > 0$ for $y \neq 0$, note that we can find $x \in \mc{M}^\circ$ such that $|\Tr(\k_{\infty,2}^{(-1)}(x) y)| > 0$ (for example take $x$ such that $\k_{\infty,2}^{(-1)}(x)$ is close to $y^*$). \\

Now by Lemma \ref{2-norm <= BMO-norm}:
\[ \|y\|_{h_1^c}  \leq \sup_{\|x\|_2 \leq 1} |\Tr(x y)| = \|y\|_2. \]
Hence we define $h_1^c$ to be the completion of $L_2^\circ(\mc{M})$ with respect to $\| \cdot \|_{h_1^c}$, and we obtain a contractive inclusion $L_2^\circ(\mc{M}) \subseteq h_1^c$. We define $h_1^r$ analogously  by taking the sup over $x$ with $\Vert x \Vert_{\BMO^c} \leq 1$.\\

We will only show that $\BMO^r \cong (h_1^c)^*$ (the other case follows similarly).
It is not hard to show that $\BMO^r \subseteq (h_1^c)^*$ contractively. Conversely, let $\psi \in (h_1^c)^*$. Then $\psi|_{L_2^\circ(\mc{M})} \in L_2^\circ(\mc{M})^*$ by Lemma \ref{2-norm <= BMO-norm}. Hence by the Riesz representation theorem there exists an $x_0 \in L_2^\circ(\mc{M})$ such that
\[
    \psi(z) = \Tr(x_0^*z)
\]
for all $z \in L_2^\circ(\mc{M})$. What remains to be shown is that $x_0^* \in \BMO^r$, with $\|x_0^*\|_{\BMO^r} \leq \|\psi\|_{(h_1^c)^*}$ (the other inequality follows from the definition of $h_1^c$). This is equivalent to requiring that $x_0 \in \BMO^c$ with $\|x_0\|_{\BMO^c} \leq \|\psi\|_{(h_1^c)^*}$\\

Fix $t > 0$. We will now use the $L_p$-modules $L_p(\mc{M} \ot_{\Phi_t} \mc{M})$ corresponding to the ucp map $\Phi_t$. Let $\Psi_p$ be the embedding of Lemma \ref{Lp to Lp module}. Then we can define the map
\[
    u_t: L_2^\circ(\mc{M}) \to L_2(\mc{M} \ot_{\Phi_t} \mc{M}), \ \ \ u_t(y) = \Psi_2(y - \Phi_t^{(2)}(y)).
\]
Now it suffices to show that
\[
    u_t(x_0) \in L_\infty(\mc{M} \ot_{\Phi_t} \mc{M}) \text{ and } \|u_t(x_0)\|_{\infty, \Phi_t} \leq \|\psi\|_{(h_1^c)^*}
\]
since then 
\begin{align*}
    \|x_0\|_{\BMO^c} &= \sup_{t \geq 0} \| \Phi_t^{(1)}(|x_0 - \Phi_t^{(2)}(x_0)|^2)\|_\infty^{1/2} \stackrel{\text{Lem. 3.13}}= \sup_{t \geq 0} \|\la u_t(x_0), u_t(x_0) \ra_1 \|_{\infty}^{1/2} \\
    &\stackrel{\text{Rem. 3.9}}= \sup_{t \geq 0} \|\k_{\infty, 1}^{(0)}(\la u_t(x_0), u_t(x_0) \ra_\infty)\|_{\infty}^{1/2} = \sup_{t \geq 0} \|\la u_t(x_0), u_t(x_0) \ra_\infty\|_{\infty}^{1/2} \\
    &= \sup_{t \geq 0} \|u_t(x_0)\|_{\infty, \Phi_t} \leq \|\psi\|_{(h_1^c)^*}.
\end{align*}
where we have used respectively the first identity  of Lemma \ref{Lp to Lp module}, the last part of Remark 3.9 and the definition of $\| \cdot \|_\infty$ in $L_1(\mc{M})$. \\

Define $\vphi_{u_t(x_0)}$ to be the dual action of $u_t(x_0)$ on $L_2(\mc{M} \ot_{\Phi_t} \mc{M})$ restricted to $\mc{M} \ot \mc{M}$, i.e.
\[
    \vphi_{u_t(x_0)}(z) := \Tr(\la u_t(x_0), z \ra_1)
\]
The goal is to prove that $u_t(x_0)$ also defines a dual action on $L_1(\mc{M} \ot_{\Phi_t} \mc{M})$. The proof is rather technical, so we contain it in a separate lemma.

\begin{lem}
Let $z \in \mc{M} \ot \mc{M}$. Then
\[
    |\vphi_{u_t(x_0)} (z)| \leq \|\psi\|_{(h_1^c)^*} \|z\|_{1, \Phi_t}
\]
In particular, $\vphi_{u_t(x_0)}$ extends to an element of $L_1(\mc{M} \ot_{\Phi} \mc{M})^*$ with $\|\vphi_{u_t(x_0)}\| \leq \|\psi\|_{(h_1^c)^*}$
\end{lem}

\begin{proof}
Let $z = \sum_j a_j \ot b_j$. Using the second identity of Lemma \ref{Lp to Lp module} and the fact that $\Phi_t^{(2)}$ is self-adjoint we have
\begin{align*}
    \Tr( \la u_t(x_0), z \ra_1) &= \sum_j \Tr( \Phi_t^{(2)}( (x_0 - \Phi_t^{(2)}(x_0))^* a_j) b_j D_\vphi^{1/2}) \\
    &= \sum_j \Tr(\Phi_t^{(2)}(x_0^* a_j)b_j D_\vphi^{1/2}) -
    \Tr(\Phi_t^{(2)}( \Phi_t^{(2)}(x_0^*)a_j)b_j D_\vphi^{1/2}) \\
    &= \sum_j \Tr( x_0^* a_j \Phi_t^{(2)}(b_j D_\vphi^{1/2})) -
    \Tr(\Phi_t^{(2)}(x_0^*) a_j \Phi_t^{(2)}(b_j D_\vphi^{1/2}) ) \\
    &= \sum_j \Tr( x_0^* a_j \Phi_t^{(2)}(b_j D_\vphi^{1/2})) -
    \Tr(x_0^* \Phi_t^{(2)}(a_j \Phi_t^{(2)}(b_j D_\vphi^{1/2}) ) ) \\
    &= \sum_j \Tr( x_0^* [ a_j \Phi_t^{(2)}(b_j D_\vphi^{1/2}) -
    \Phi_t^{(2)}(a_j \Phi_t^{(2)}(b_j D_\vphi^{1/2}) )]) \\
    &= \Tr(x_0^*u_t^*(z)).
\end{align*}
Thus $u_t^*(z) := \sum_j a_j \Phi_t^{(2)}(b_j D_\vphi^{1/2}) - \Phi_t^{(2)}(a_j \Phi_t^{(2)}(b_j D_\vphi^{1/2})) \in L_2(\mc{M})$.\\

We are done if we can prove that $\|u_t^*(z)\|_{h_1^c} \leq \|z\|_{1, \Phi_t}$. However, we do not even have $u_t^*(z) \in L_2^\circ(\mc{M})$ in general, so this will not be possible. To circumvent this, let $\pi$ be the projection $L_2(\mc{M}) \to L_2^\circ(\mc{M})$. Then $\pi$ is self-adjoint and $\pi(x_0) = x_0$, hence
\[
    \Tr(x_0^* u_t^*(z)) = \Tr(x_0^* \pi(u_t^*(z))).
\]
We claim that $\|\pi(u_t^*(z))\|_{h_1^c} \leq \|z\|_{1, \Phi_t}$. Indeed, by \eqref{dual GNS action} and Remark \ref{equality of brackets}:
\begin{align*}
    \|\pi(u_t^*(z))\|_{h_1^c} &= \sup_{\|y\|_{\BMO^r} \leq 1}  |\Tr(y \pi(u_t^*(z)))|
    = \sup_{\|y\|_{\BMO^c} \leq 1}  |\Tr(y^* \pi(u_t^*(z)))|  \\
    &= \sup_{\|y\|_{\BMO^c} \leq 1}  |\Tr( \la u_t(y), z \ra_1)|
    = \sup_{\|y\|_{\BMO^c} \leq 1}  |\Tr( \la u_t(y), z \ra_{\infty, 1})| \\
    &\leq \sup_{\|y\|_{\BMO^c} \leq 1} \|z\|_{1, \Phi_t} \|u_t(y)\|_{\infty, \Phi_t} = \|z\|_{1, \Phi_t}.
\end{align*}
It follows that indeed
\[
    |\vphi_{u_t(x_0)}(z)| = |\Tr(x_0^*u_t^*(z))| \leq  \sup_{\|h\|_{h_1^c} \leq 1} |\Tr(x_0^* h)| \|z\|_{1, \Phi_t} = \|\psi\|_{(h_1^c)^*} \|z\|_{1, \Phi_t}
\]
\end{proof}

Now through our duality result of Proposition \ref{duality GNS module},  $u_t(x_0) \in L_\infty(\mc{M} \ot_{\Phi} \mc{M})$ and
\[
    \|x_0\|_{\BMO^c} = \sup_{t \geq 0} \|u_t(x_0)\|_{\infty, \Phi_t}  = \sup_{t \geq 0} \sup_{\|z\|_{1, \Phi_t} \leq 1} |\Tr(\la u_t(x_0), z \ra_{\infty, 1})| \leq \|\psi\|_{(h_1^c)^*}.
\]
This shows that indeed $\BMO^r \cong (h_1^c)^*$.
\end{proof}

Note that this also proves that $\| \cdot \|_{\BMO^c}$, $\| \cdot \|_{\BMO^r}$ satisfy the triangle inequality and that $\BMO^c$ and $\BMO^r$ are Banach spaces. Hence $(\BMO^c, \BMO^r)$ is a well-defined compatible couple and the intersection space $\BMO$ is also a well-defined Banach space:
\begin{cor}
$\BMO(\mc{M}, \Phi)$,  $\BMO^c(\mc{M}, \Phi)$ and  $\BMO^r(\mc{M}, \Phi)$ are Banach spaces.
\end{cor}

\begin{rem} \label{Rem=Weak*TopologyForBMO}
In the absence of a predual for $\BMO$, we will define a ``weak-$\ast$ topology" in a different way, namely as the locally convex topology inherited from the topologies $\s(\BMO^c, h_1^r)$ and $\s(\BMO^r, h_1^c)$. By slight abuse of notation, we will call this the weak-$\ast$ topology. More precisely, recall that for $x \in \BMO$, we denoted by $x_c \in \BMO^c$ and $x_r \in \BMO^r$ those elements for which $x = \k^{(-1)}_{2,1}(x_c) = \k^{(1)}_{2,1}(x_r)$. Then we say that a net $x^\l \in \BMO$ converges to $x \in \BMO$ in the weak-$\ast$ topology if $x_c^\l \to x_c$ in the weak-$\ast$ topology of $\BMO^c$ and $x_r^\l \to x_r$ in the weak-$\ast$ topology of $\BMO^r$.
\end{rem}

\bigskip

\subsection{Interpolation for BMO space}\label{SubSect=Interpolation}
In this section we show that \cite[Theorem 4.5]{Cas18} holds again for the current definiton of $\BMO$. Similar to how \cite[Theorem 4.5]{Cas18} is proved, the proof is a mutatis mutandis copy of the methods in \cite[Section 3]{Cas18} provided that conditional expectations extend to a contraction on BMO. In other words, we must show that \cite[Lemma 4.3]{Cas18} still holds in the current setup. This is done in Proposition \ref{Prop=1complement} below. We start with some auxiliary lemmas that could be of independent interest.

  Let us state some preliminary facts. By \cite[Theorem II.36]{Terp81}, a standard form for $\mc{M}$ is $(\mc{M}, L_2(\mc{M}), J, L_2^+(\mc{M}))$, where $J$ is the conjugation operator. Hence we will consider $\mc{M}$ as a von Neumann subalgebra of $\mc{B}(L_2(\mc{M}))$ by left multiplication.
With an inclusion of von Neumann algebras $\mc{M}_1 \subseteq \mc{M}$ we mean a unital inclusion, meaning that the unit of $\mc{M}_1$ equals the unit of $\mc{M}$. It is a well known fact that $\mc{M}_1$ admits a $\varphi$-preserving conditional expectation if and only if $\sigma_t^\varphi(\mc{M}_1) = \mc{M}_1$ for all $t \in \mathbb{R}$, see \cite[Theorem IX.4.2]{Takesaki2}. If $\mc{E}$ is a $\vphi$-preserving conditional expectation, then we can use Proposition \ref{prop extension to Lp} to extend it to a contraction $\mc{E}^{(p)}: L_p(\mc{M}) \to L_p(\mc{M})$, which can be checked to land in $L_p(\mc{M}_1)$.

\begin{lem}\label{Lem=DualSelf}
Let $\mc{M}_1 \subseteq \mc{M}$ be a  von Neumann subalgebra that admits a $\varphi$-preserving  conditional expectation $\mathcal{E}$.  Then for $x \in L_1(\mc{M})$ and $y \in \mc{M}$ we have
\[
{\rm Tr}(x \mathcal{E}(y)) =
{\rm Tr}( \mathcal{E}^{(1)}(x) y).
\]
\end{lem}
\begin{proof}
If $x = D_\varphi x'$ with $x'  \in \mc{M}$ we have since $\mathcal{E}^{(1)}$ is  $\Tr$-preserving,
\[
\begin{split}
& {\rm Tr}(x \mathcal{E}(y)) =
{\rm Tr}(\mathcal{E}^{(1)}(x \mathcal{E}(y))) =
{\rm Tr}( D_\varphi \mathcal{E}(x' \mathcal{E}(y))) =
{\rm Tr}( D_\varphi \mathcal{E}(x') \mathcal{E}(y)  ) \\
= &
{\rm Tr}(  D_\varphi  \mathcal{E}( \mathcal{E}(x') y)  )
=
{\rm Tr}(   \mathcal{E}^{(1)}(D_\varphi  \mathcal{E}(x') y)  )=
{\rm Tr}( \mathcal{E}^{(1)}(x) y).
\end{split}
\]
For general $x \in L_1(\mc{M})$ the statement follows by approximation.
\end{proof}

The following lemma is a  variation of the Kadison-Schwarz inequality.

\begin{lem}\label{Lem=KS}
Let $\mc{M}_1 \subseteq \mc{M}$ be a  von Neumann subalgebra that admits a $\varphi$-preserving   conditional expectation $\mathcal{E}$.  Then for $x \in L_2(\mc{M})$  we have the following inequality in  $L_1(\mc{M})$,
\[
\mathcal{E}^{(2)}(x)   \mathcal{E}^{(2)}(x)^\ast \leq \mathcal{E}^{(1)}(x   x^\ast).
\]
\end{lem}

\begin{proof}
Naturally $L_2(\mc{M}_1) \subseteq L_2(\mc{M})$ is a closed subspace and we have that $\mathcal{E}^{(2)}:  L_2(\mc{M}) \rightarrow  L_2(\mc{M}_1)$ is the orthogonal projection onto this subspace, see \cite[Proof of Theorem IX.4.2]{Takesaki2}. $L_2(\mc{M}_1)$ is an invariant subspace for $\mc{M}_1$. Therefore $\mc{M}_1$ commutes with both  $\mathcal{E}^{(2)}$ and $1-\mathcal{E}^{(2)}$. Hence, for $y \in \mc{M}_1$ and $x \in L_2(\mc{M})$ we have
\[
\langle  \mathcal{E}^{(2)}(x), y \mathcal{E}^{(2)}(x) \rangle + \langle (1-\mathcal{E}^{(2)})(x), y (1-\mathcal{E}^{(2)})(x) \rangle  = \langle  x, y x \rangle.
\]
And so for $y \in \mc{M}^+$ we have
\begin{equation}\label{Eqn=L1Pos}
\Tr(  \mathcal{E}(y) \mathcal{E}^{(2)}(x)   \mathcal{E}^{(2)}(x)^\ast    )  =  \langle  \mathcal{E}^{(2)}(x), \mathcal{E}(y) \mathcal{E}^{(2)}(x) \rangle \leq \langle   x,   \mathcal{E}(y) x \rangle
= \Tr(      \mathcal{E}(y) x   x^\ast).
\end{equation}
We further have by Lemma \ref{Lem=DualSelf},
\[
\Tr(      \mathcal{E}(y) x   x^\ast) = \Tr( y  \mathcal{E}^{(1)}(x   x^\ast)),
\]
and  since $\mc{E}^{(1)}$ is a projection onto $L_1(\mc{M}_1)$
\[
\Tr(  \mathcal{E}(y) \mathcal{E}^{(2)}(x)   \mathcal{E}^{(2)}(x)^\ast    )
 = \Tr( y  \mathcal{E}^{(1)}(\mathcal{E}^{(2)}(x)   \mathcal{E}^{(2)}(x)^\ast)    ) =
\Tr( y  \mathcal{E}^{(2)}(x)   \mathcal{E}^{(2)}(x)^\ast   ).
\]
Therefore \eqref{Eqn=L1Pos} shows that we have the following Kadison-Schwarz type inquality,
\[
\mathcal{E}^{(2)}(x)   \mathcal{E}^{(2)}(x)^\ast \leq \mathcal{E}^{(1)}(x   x^\ast).
\]
\end{proof}

\begin{lem}\label{Lem=Comparison}
Let $\omega \in \mc{M}_\ast^+$. The following are equivalent:
\begin{enumerate}
\item\label{Item=ComparisonI} We have  $\omega \leq \varphi$.
\item\label{Item=ComparisonII} There exists $x \in \mc{M}$ with $0 \leq x \leq 1$ such that $D_\varphi^{\frac{1}{2}} x D_\varphi^{\frac{1}{2}} = D_\omega$.
\end{enumerate}
\end{lem}

\begin{proof}
For \eqref{Item=ComparisonI} $\Rightarrow$ \eqref{Item=ComparisonII}, consider the map
\[
T: L_2(\mc{M}) \rightarrow  L_2(\mc{M}): D_\varphi^{\frac{1}{2}} x \mapsto D_\omega^{\frac{1}{2}} x, \qquad x \in \mc{M}.
\]
From the fact that $\omega \leq \varphi$ it follows that $T$ is a well-defined contraction. Moreover, we claim that $T \in \mc{M}$. Indeed,  the commutant of $\mc{M}$ acting on $L_2(\mc{M})$ is given by $J \mc{M} J$ where $J: \xi \mapsto \xi^\ast$ is the modular conjugation. Then it follows that for $x, y \in \mc{M}$  we have
\[
T J y J  D_\varphi^{\frac{1}{2}} x = T  D_\varphi^{\frac{1}{2}} x y^\ast =  D_\omega^{\frac{1}{2}} x y^\ast =  J y J  T(D_\varphi^{\frac{1}{2}} x).
\]
Now set $x = T^\ast T \in \mc{M}$ so that $0 \leq x \leq 1$. We have  $T D_\varphi^{\frac{1}{2}} = D_\omega^{\frac{1}{2}}$ so that $(D_\varphi^{\frac{1}{2}} T^\ast) (T D_\varphi^{\frac{1}{2}}) = D_\omega$.

 The implication  \eqref{Item=ComparisonII} $\Rightarrow$ \eqref{Item=ComparisonI} follows as for $y \in \mc{M}$ we have
\[
\begin{split}
&\omega(y y^\ast) = {\rm Tr}(D_\omega y y^\ast) = \Tr(y^\ast D_\varphi^{\frac{1}{2}} x D_\varphi^{\frac{1}{2}} y  ) =
\langle D_\varphi^{\frac{1}{2}} y  , x D_\varphi^{\frac{1}{2}} y   \rangle \\
\leq & \langle D_\varphi^{\frac{1}{2}} y  ,  D_\varphi^{\frac{1}{2}} y   \rangle =
\Tr( y^\ast   D_\varphi y )  = \varphi(y y^\ast).
\end{split}
\]
\end{proof}

\begin{lem}\label{Lem=Majorize}
Let $a, b \in L_1(\mc{M})^+$ and suppose that $a \leq b$ and $b = D_\varphi^{\frac{1}{2}} x_b D_\varphi^{\frac{1}{2}}$ with $x_b \in \mc{M}^+$. Then there exists $x_a \in \mc{M}^+$ such that $a =  D_\varphi^{\frac{1}{2}} x_a D_\varphi^{\frac{1}{2}}$. Moreover $x_a \leq x_b$.
\end{lem}
\begin{proof}
Let $\varphi_a$ and $\varphi_b$ be in $\mc{M}_\ast^+$ such that $D_{\varphi_a} = a$ and $D_{\varphi_b} = b$. The assumptions and
Lemma \ref{Lem=Comparison} imply that $\varphi_b \leq \Vert x_b \Vert \varphi$.  We find that $\varphi_a \leq \varphi_b \leq \Vert x_b \Vert \varphi$. Therefore Lemma \ref{Lem=Comparison} implies that there exists $x_a \in \mc{M}$ with $0 \leq x_a \leq \Vert x_b \Vert$ such that $a = D_\varphi^{\frac{1}{2}} x_a D_\varphi^{\frac{1}{2}}$. We have moreover $x_a \leq x_b$ since $a \leq b$ implies that for $y \in \mc{M}$,
\[
\begin{split}
& \langle  D_\varphi^{\frac{1}{2}} y ,  x_a D_\varphi^{\frac{1}{2}} y \rangle = \Tr(   y^\ast D_\varphi^{\frac{1}{2}}   x_a D_\varphi^{\frac{1}{2}} y ) =
\Tr(   D_\varphi^{\frac{1}{2}}   x_a D_\varphi^{\frac{1}{2}} y y^\ast )
=  \Tr(a y y^\ast) \\ \leq & \Tr( b y y^\ast) =  \Tr(   D_\varphi^{\frac{1}{2}}   x_b D_\varphi^{\frac{1}{2}} y y^\ast ) =
\langle  D_\varphi^{\frac{1}{2}} y ,  x_b D_\varphi^{\frac{1}{2}} y \rangle.
\end{split}
\]
 \end{proof}

\begin{prop}\label{Prop=1complement}
Let $\mc{M}_1 \subseteq \mc{M}$ be a von Neumann subalgebra that admits a $\varphi$-preserving  conditional expectation $\mathcal{E}$. Let $\Phi = (\Phi_t)_{t \geq 0}$ be a Markov semi-group on $\mc{M}$ that preserves $\mc{M}_1$.  Then we have isometric 1-complemented inclusions
\[
\BMO(\mc{M}_1, \Phi) \subseteq \BMO(\mc{M}, \Phi).
\]
\end{prop}
\begin{proof}
That the isometric inclusion exists is clear from the definitions. We have to prove that the inclusion is 1-complemented.
 For $t \geq 0$ and $x \in  \BMO_\Phi^c(\mc{M}) \subseteq L_2^\circ(\mc{M})$ we have the following (in)equalities in $L_1(\mc{M})$ by Lemma \ref{Lem=KS},
\[
\begin{split}
&  \vert \mathcal{E}^{(2)}(x) - \Phi_t^{(2)}(\mathcal{E}^{(2)}(x)) \vert^ 2  =    \mathcal{E}^{(2)}(x - \Phi_t^{(2)}(x) )^\ast \mathcal{E}^{(2)}(x - \Phi_t^{(2)}(x) ) \\
& \leq      \mathcal{E}^{(1)}   (    (x - \Phi_t^{(2)}(x) )^\ast  (x - \Phi_t^{(2)}(x) )     )  ).
\end{split}
\]
As $\Phi^{(1)}_t$ preserves positivity and commutes with $\mathcal{E}^{(1)}$,
\begin{equation}\label{Eqn=IntermedExp}
\Phi_t^{(1)}( \vert \mathcal{E}^{(2)}(x) - \Phi_t^{(2)}(\mathcal{E}^{(2)}(x)) \vert^ 2  ) \leq    \mathcal{E}^{(1)} ( \Phi_t^{(1)}  (    (x - \Phi_t^{(2)}(x) )^\ast  (x - \Phi_t^{(2)}(x) )     )  ).
\end{equation}
By assumption we may write
\[
 \Phi_t^{(1)}  (    (x - \Phi_t^{(2)}( x) )^\ast  (x - \Phi_t^{(2)}(x) ) = \kappa_{\infty, 1}^{(0)}(x'_t),
 \]
for some $x'_t  \in \mc{M}$. So the right hand side of \eqref{Eqn=IntermedExp} equals $\kappa_{\infty, 1}^{(0)}( \mathcal{E}(x_t'))$.
By Lemma \ref{Lem=Majorize} it follows that there exists $x''_t  \in \mc{M}$ with $0 \leq x''_t \leq  \mc{E}(x'_t)$ such that
\[
\Phi_t^{(1)}( \vert \mathcal{E}^{(2)}(x) - \Phi_t^{(2)}(\mathcal{E}^{(2)}(x)) \vert^ 2  )= \kappa_{\infty, 1}^{(0)}(x''_t).
\]
Taking norms we have
\[
\Vert \mathcal{E}^{(2)}(x) \Vert_{\BMO^c} =  \sup_{t \geq 0} \Vert x''_t \Vert_\infty \leq  \sup_{t \geq 0} \Vert \mathcal{E}(x'_t) \Vert_\infty \leq   \sup_{t \geq 0} \Vert x'_t \Vert_\infty
= \Vert x \Vert_{\BMO^c}.
\]
The row BMO-estimate and the BMO-estimate follow similarly.
\end{proof}

We may now conclude the following theorem. The proof (based on the Haagerup reduction method) follows exactly as in \cite[Sections 3 and 4]{Cas18} where \cite[Lemma 4.3]{Cas18} needs to be replaced by Proposition \ref{Prop=1complement}. Note that in the statement of \cite[Theorem 4.5]{Cas18} the standard Markov dilation must be modular as well (this is a misprint in the text of \cite{Cas18}).

\begin{thm} \label{bmo interpolation thm}
Let $\Phi$ be a $\varphi$-modular Markov semigroup on a $\s$-finite von Neumann algebra $(\mc{M}, \varphi)$ admitting a modular standard Markov dilation. Then for all $1 \leq p < \infty, 1 < q < \infty$,
\[
    [\BMO(\mc{M}, \Phi), L_p^\circ(\mc{M})]_{1/q} \approx_{pq} L^\circ_{pq}(\mc{M}).
\]
Here $\approx_{pq}$ means that the Banach spaces are isomorphic and the norm of the isomorphism in both directions can be estimated by an absolute constant times $pq$.
\end{thm}

\section{$L_p$-boundedness of BMO-valued Fourier-Schur multipliers on $SU_q(2)$} \label{section SUq(2)}

  In this section we prove that Fourier-Schur multipliers on $SU_q(2)$ of a certain form extend to the non-commutative $L_p$ spaces corresponding to $SU_q(2)$. We first introduce compact quantum groups, $SU_q(2)$ and give the definition of Fourier-Schur multipliers. Then we prove the endpoint estimates we need for complex interpolation.

\subsection{BMO spaces of the torus} \label{Sect=torus}
Define trigonometric functions
\[
\z_k: \T \to \T: z \mapsto z^k, \qquad k \in \mathbb{Z}.
\]
Set the $\ast$-algebra $\text{Pol}(\T) := \Span\{\z_k: k \in \Z\}$.
For $m \in \ell_\infty(\Z)$ let $T_m: L_2(\T) \to L_2(\T)$ be the Fourier multiplier defined by $T_m(\zeta_k) = m(k) \zeta_k, k \in \mathbb{Z}$. For $t \geq 0$ let  $h_t \in \ell_\infty(\Z)$ be given by $h_t(k) = e^{-tk^2}$. Then the maps $T_{h_t}$ are well-known to define a Markov semigroup on the von Neumann algebra $L_\infty(\T)$ (as they are restrictions of the Heat semi-group on $L_\infty(\mathbb{R})$).  We use the shorthand notation
\[
\BMO(\T) := \BMO(L_\infty(\T), (T_{h_t})_{t \geq 0}).
\]

Let $m \in \ell_\infty(\Z)$ be such that $m(0) = 0$. Then as $t \rightarrow \infty$,
\[
    \|T_{h_t}(T_m \zeta_k)\|_\infty = e^{-tk^2} |m(k)| \| \zeta_k \|_\infty \to 0.
\]
So $T_m$ maps $\rm{Pol}(\T)$ to $L_\infty^\circ(\T)$. 

\subsection{Compact quantum groups} \label{Sect=QGIntro}

For the theory of compact quantum groups we refer to \cite{Wor98} or the notes \cite{MD98} which follows the same lines.

\begin{defi}
A compact quantum group $\GG = (C(\GG), \Delta)$ consists of a unital C$^\ast$-algebra $C(\GG)$ and a unital $\ast$-homomorphism  $\Delta: C(\GG) \rightarrow C(\GG) \otimes_{{\rm min}} C(\GG)$ called the comultiplication such that $(\Delta \otimes \iota) \circ \Delta = (\iota \otimes \Delta) \circ \Delta$ (coassociativity) and such that both $\Delta(C(\GG)) (C(\GG) \otimes 1)$ and $\Delta(C(\GG)) (1 \otimes C(\GG))$ are dense in $C(\GG) \otimes_{\min} C(\GG)$. Here $\iota: C(\GG) \rightarrow C(\GG)$ is the identity map.
\end{defi}

A finite dimensional (unitary) corepresentation is a unitary $u \in C(\GG) \otimes M_n(\mathbb{C})$ such that $(\Delta \otimes \id) (u) = u_{13} u_{23}$ where $u_{23} =  1 \otimes u$ and $u_{13}$ is the flip applied to the first two tensor legs of $u_{23}$. All corepresentations are assumed to be unitary. The elements $(\id \otimes \omega)(u) \in C(\GG)$ with  $\omega \in M_n(\mathbb{C})^\ast$ are called matrix coefficients. The span of all matrix coefficients is a $\ast$-algebra called $\Pol(\GG)$.   $\Delta$ maps $\Pol(\GG)$ to $\Pol(\bG) \otimes \Pol(\bG)$.

 Here we shall mainly be concerned with the quantum group $SU_q(2)$ and we shall introduce further structure such as Haar states and von Neumann algebras for this case only.

\subsection{Introduction $SU_q(2)$}
Let $\GG_q := SU_q(2)$ with $q \in (-1,1) \backslash \{ 0 \}$. It was introduced by Woronowicz in \cite{Wor87}. Its $C^*$-algebra is the one generated by the operators $\a, \g$ on the Hilbert space $\mc{H} = \ell_2(\N) \ot_2 \ell_2(\Z)$ given by 
\[
\begin{split}
\a (e_i \ot f_j) = & \sqrt{1 - q^{2i}} e_{i-1} \ot f_j, \\
\g (e_i \ot f_j) = & q^i e_i \ot f_{j+1}.
\end{split}
\]
where $e_i \ot f_j, i \in \mathbb{N}, j \in \mathbb{Z}$ are the basis vectors of $\mc{H}$. The operators $\a, \g$ satisfy the following relations:
\begin{table}[H]
\centering
\begin{tabular}{ccccc}
$\g^*\g = \g\g^*,$ && $\a\g = q\g\a,$ && $\a\g^* = q\g^*\a,$\\
 $\a^*\a + \g^*\g = I,$ && $\a\a^* + q^2\g^*\g = I.$ &&
\end{tabular}
\end{table}
\noindent The comultiplication is given by
\[
    \Delta(\a) = \a \ot \a - q \g^* \ot \g, \ \ \ \ \ \ \Delta(\g) = \g \ot \a + \a^* \ot \g.
\]

We define $L_\infty(\GG_q) =  \la \a, \g \ra'' \subseteq \mc{B}(\mc{H})$. The corresponding noncommutative $L_p$-spaces are written as $L_p(\GG_q)$. We also define $\rm{Pol}(\GG_q) \subseteq L_\infty(\GG_q)$ to be the $\ast$-algebra generated by $\a, \g$.  This is equivalent to the definition given in Section \ref{Sect=QGIntro}. It is the linear span of elements $\a^k \g^l (\g^*)^m$, $k \in \Z, l,m \in \mathbb{N}$, where we set $\a^k = (\a^*)^{|k|}$ in case $k < 0$. Obviously, $\rm{Pol}(\GG_q)$ is weakly (or weak-$\ast$) dense in $L_\infty(\GG_q)$.


The Haar state on $L_\infty(\GG_q)$ is given by the following formula:
\begin{equation} \label{haar state} \vphi(x) = (1-q^2) \sum_{k \in \N} q^{2k} \la e_k \ot f_0, x (e_k \ot f_0) \ra. \end{equation}
See \cite[Appendix A1]{Wor87b} for the complete calculation.
Note that $\vphi(\a^k \g^l(\g^*)^m)$ is non-zero if and only if $k = 0, l = m$. It is also faithful, as follows for instance from \eqref{haar state}.

The modular automorphism group is given by
\begin{equation} \label{sigma_t} \s_t^\vphi(\a^k \g^l (\g^*)^m) = q^{-itk} \a^k \g^l (\g^*)^m. \end{equation}
This can be derived from \cite[Theorem VIII.3.3]{Takesaki2}, where the $u_t$ from the theorem is equal to $(\g^*\g)^{it}$ and the $\psi$ is a trace.

\begin{rem}
The above definition of $L_\infty(\GG_q)$ is not the standard way to define the von Neumann algebra; usually this would be the double commutant within the GNS-representation corresponding to the Haar state $\phi$. However, these von Neumann algebras are isomorphic, although they are not unitarily isomorphic.
\end{rem}

\bigskip

\subsection{Fourier-Schur Multipliers on $SU_q(2)$} \label{SubSect=FourierSchur}

\begin{defi}
Let $\mathbb{G}$ be a compact quantum group and $T: \rm{Pol}(\mathbb{G}) \to \rm{Pol}(\mathbb{G})$ a linear map.  We call $T$ a  Fourier-Schur multiplier if the following condition holds.
 Let $u$ be any finite dimensional corepresentation on $\mc{H}$. Then there exists an orthogonal basis $e_i$ such that if  $u_{i,j}$ are the matrix coefficients with respect to this basis, then there exist numbers $c_{i,j} := c_{i,j}^u   \in \C$ such that
\[
    T u_{i,j} = c_{i,j} u_{i,j}.
\]
In this case $(c_{i,j}^u )_{i,j,u}$ is called the symbol of $T$.
\end{defi}

\begin{rem}
If $\GG$ comes from a classical abelian group $G$, i.e. if all irreducible corepresentations are one-dimensional, then the above definition coincides with the definition of a classical Fourier multiplier. In general, we see that $T = \mc{F} S \mc{F}^{-1}$, where $S$ is a Schur multiplier. Hence the name `Fourier-Schur multiplier'.
\end{rem}

We will construct Fourier-Schur multipliers from Fourier multipliers on the torus $\T \subseteq \C$.
 We assume henceforth that $m \in \ell_\infty(\Z)$ with $m(0) = 0$ such that $T_m: L_\infty(\T) \to \BMO(\T)$ is completely bounded. In the remainder of this section, we will consider the map
\begin{equation}\label{Eqn=FourierSchur}
    \tilde{T}_m: \text{Pol}(\GG_q) \to \text{Pol}(\GG_q), \ \ \ \a^k \g^l(\g^*)^m \mapsto m(k) \a^k \g^l (\g^*)^m
\end{equation}
We will see after the next subsection
that $\tilde{T}_m$ is indeed a Fourier-Schur multiplier. We remark that the symbol $m$ is used both as an element of $\ell_\infty(\Z)$ and a power of $\g^*$; the context will always make clear which is meant. \\


We introduce at this point the Markov semigroup that we will use to define the BMO space:
\[
\Phi_t(\a^k \g^l (\g^*)^m ) = e^{-tk^2} \a^k \g^l (\g^*)^m, \qquad k \in \mathbb{Z}, l,m \in \mathbb{N}, t \geq 0.
\]
We will only prove in Section \ref{SubSect=BMODef} that the maps $\Phi_t$ extend to form a Markov semigroup on $L_\infty(\GG_q)$. However, for the sake of exposition it will be convenient to already define the corresponding spaces $L_p^\circ(\GG_q)$ as in Section \ref{SubSect=BMOintro}. \\

The final goal is to prove that this map extends boundedly to $L_p(\GG_q) \to L_p^\circ(\GG_q)$ for all $p \geq 2$. We do this through complex interpolation (Riesz-Torin). This requires 3 steps: (1) a lower endpoint estimate; (2) an upper endpoint estimate involving $\BMO$ spaces and (3) the construction of a Markov dilation in order to apply Theorem \ref{bmo interpolation thm}.

We treat the Markov dilation in Appendix \ref{markov dilation section}. The remainder of this section is devoted to the endpoint estimates. \\

Similarly to the torus, we have

\begin{lem} \label{image in circ}
Let $1 \leq p \leq \infty$. Then $\k_{\infty, p}^{(1)} \circ \tilde{T}_m$ maps $\rm{Pol}(\GG_q)$ to $L_p^\circ(\GG_q)$.
\end{lem}

\begin{proof}
Let $x = \a^k \g^l (\g^*)^m$. For $k = 0$, we have $\tilde{T}_m(x) = 0 \in L_p^\circ(\GG_q)$. Now assume $|k| > 0$. Then for any $1 \leq p \leq \infty$, we have as $t \rightarrow \infty$,
\[
    \|\Phi_t^{(p)}(\k_{\infty,p}^{(1)}(\tilde{T}_m x))\|_p = \|\k_{\infty, p}^{(1)}(\Phi_t(\tilde{T}_m(x)))\|_p = |m(k) e^{-t k^2}| \|\k_{\infty, p}^{(1)}(x)\|_p \to 0.
\]
Since $\rm{Pol}(\GG_q)$ is the span of elements $\a^k \g^l (\g^*)^m$, the result follows by linearity. (Note that for $p = \infty$, the $\s$-weak convergence follows from norm convergence.)
\end{proof}

\bigskip

\subsection{$L^2$-estimate} \label{Sect=L2_estimate}

In this subsection we prove that \eqref{Eqn=FourierSchur} extends to a bounded map $L_2(\GG_q) \to L_2(\GG_q)$. At the same time we prove (essentially) that it defines a Fourier-Schur multiplier.
The main ingredient will be the Peter-Weyl decomposition of $\GG_q$ (see \cite[Theorem 4.17]{KlimykSchmudgen}) we shall summarize now.


A complete set of mutually inequivalent irreducible corepresentations of $\GG_q$ can be constructed as follows. They are labeled by half integers $l \in  \frac{1}{2} \mathbb{N}$. Consider the vector space of linear combinations of the homogeneous polynomials in $\a, \g$ of degree $2l$. For some specific constants $C_{l,k,q}$, we define basis vectors as follows:
\begin{equation}\label{basis vectors}
    g^{(l)}_k = C_{l,k,q} \a^{l-k} \g^{l+k}, \ \   k = -l, -l+1, \ldots, l.
\end{equation} 
The precise value of the constant $C_{l,k,q}$ can be found in \cite[Chapter 4.2.3]{KlimykSchmudgen}; it is of little importance to us. Next, we define the matrix $u^{(l)} \in \Pol(\bG_q) \ot M_{2l+1}(\C)$ by
\[ \D(g^{(l)}_k)= \sum_{i=-l}^l u^{(l)}_{k,i} \ot g_i^{(l)}.
\]
The Peter-Weyl theorem now takes the following form from which we derive the main result of this subsection in Proposition \ref{Prop=Coefficients}.

\begin{lem}[Proposition 4.16 and Theorem 4.17 of \cite{KlimykSchmudgen}]\label{Lem=PeterWeyl}
The matrix coefficients of $u^{(l)} \in M_{2l+1}(L_\infty(\GG_q))$ are a linear basis for $\rm{Pol}(\GG_q)$ satisfying the orthogonality relations
\[ \vphi((u^{(l)}_{i,j})^* u^{(k)}_{r,s}) = C_i^{(l)} \d_{l, k} \d_{i, r} \d_{j, s} . \]
for some constants $C_i^{(l)} \in \C$.
\end{lem}


%
%

\begin{prop}\label{Prop=Coefficients}
The $u^{(l)}_{i,j}$ form an orthogonal basis of eigenvectors for the map $\tilde{T}_m$ defined in \eqref{Eqn=FourierSchur} with eigenvalues $m(-i-j)$.
\end{prop}

\begin{proof}
To prove this, we will calculate an explicit expression for the matrix elements $u^{(l)}_{i,j}$. With our notation $\a\a^{-1} = \a\a^* = 1- q^2\g^*\g$. Hence,
\begin{align*}
    \a^k (\a^*)^k &= \a^{k-1} (1 - q^2\g^*\g) (\a^*)^{k-1} = (1 - q^{2k} \g^*\g) \a^{k-1} (\a^*)^{k-1} \\
    &= \dots = (1 - q^{2k} \g^*\g) (1 - q^{2k-2}\g^*\g) \dots (1 - q^2\g^*\g) =: (q^2\g^*\g; q^2)_k.
\end{align*}
The notation $(a; b)_k$ is known as the Pochhammer symbol. We define $\sbinom{k}{i}_q$ to be the {\em q-binomial coefficients} from \cite[Section 2.1.2]{KlimykSchmudgen}. They satisfy the formula 
\[
    (v + w)^k = \sum_{i = 0}^k  \sbinom{k}{i}_{q^{-1}}\ v^i w^{k - i}.
\]
for $v,w$ satisfying $vw = qwv$. Below we will use this formula on both tensor legs simultaneously, which means that the subscript of the $q$-binomial coefficient becomes $q^{-2}$. Thus:

\begin{align*}
\D(g^{(l)}_i) &= C_{l,i,q} \D(\a^{l-i} \g^{l+i}) = C_{l,i,q} \D(\a)^{l-i} \D(\g)^{l+i} \\
&= C_{l,i,q} (\a \ot \a - q\g^* \ot \g)^{l-i} (\g \ot \a + \a^* \ot \g)^{l+i} \\
&= C_{l,i,q} \left(\sum_{a= 0}^{l-i} (-q)^{l-i-a} \sbinom{l-i}{a}_{q^{-2}}\ \a^{a} (\g^*)^{l-i-a} \ot \a^{a} \g^{l-i-a}\right) \\
& \qquad \qquad \times \qquad \left(\sum_{s = 0}^{l+i} \sbinom{l+i}{s}_{q^{-2}}\ \g^{s} (\a^*)^{l+i-s} \ot \a^s\g^{l+i-s} \right) \\
&= C_{l,i,q} \sum_{a=0}^{l-i} \sum_{s=0}^{l+i} C'_{a, s} \a^{a+s-l-i} (\g^*)^{l-i-a} \g^s P_{a,s}(\g^*, \g) \ot \a^{a+s} \g^{2l-a-s} \\
\end{align*}
where $C'_{a, s} := C'_{l,i,q,a,s} = (-q)^{l-i-a} q^{(l+i-s)(s+l-i-a) - s(l-i-a)} \sbinom{l-i}{a}_{q^{-2}} \sbinom{l+i}{s}_{q^{-2}}$ and $P_{a,s}(\g^*, \g) := P_{l,i,q,a,s}(\g^*, \g)$ is some polynomial in the variables $\g^*$, $\g$ depending on the minimum value of $\{a, l+i-s\}$. If the minimum value is $l+i-s$ then $P_{a,s}(\g^*, \g)= (q^2\g^*\g; q^2)_{\min(a, l+i-s)}$; if it is $a$ then the Pochhammer symbol appears instead to the left of $\a^{a+s-l-i}$, so after interchanging we obtain extra powers of $q$ in the terms of the polynomial. \\

Next, we substitute $s$ by $j$ where $j = l-a-s$ and set $P'_{a,j}(\g^*, \g) := (\g^*)^{l-i-a} \g^{l-j-a}  P_{a,l-j-a}(\g^*, \g)$, $C''_{a, j} := C'_{a, l-j-a}$ with slight abuse of notation. This gives:
\begin{align*}
\D(g^{(l)}_i) &= C_{l,i,q} \sum_{a=0}^{l-i}\sum_{j = -a-i}^{l-a}  C''_{a,j} \a^{-(i+j)} P'_{a,j}(\g^*, \g) \ot \a^{l-j}\g^{l+j} \\
&=  C_{l,i,q} \sum_{j = -l}^l \sum_{a = \max\{0, -i-j\}}^{\min\{l-i, l-j\}} C''_{a,j}  \a^{-(i+j)} P_{a,j}'(\g^*, \g) \ot C_{l,j,q}^{-1} g_j^{(l)}.
\end{align*}
Hence we find
\begin{equation} \label{expr for t} u^{(l)}_{i,j} = \a^{-(i+j)} \cdot C_{l,i,q} C_{l,j,q}^{-1} \sum_a  C''_{a,i,j,l,q} P_{a,i,j,l,q}'(\g^*, \g). \end{equation}

Now since the only power of $\a$ that occurs in \eqref{expr for t} is $\a^{-(i+j)}$, the $u^{(l)}_{i,j}$ are eigenvectors for the maps $\tilde{T}_m$.
\end{proof}

\begin{cor}
The map \eqref{Eqn=FourierSchur} is a Fourier-Schur multiplier for $\GG_q$ with symbol $(m(-i-j))_{i,j,l}$ where $l \in \frac{1}{2} \mathbb{N}$ indexes the corepresentation and $1 \leq i,j \leq 2l +1$.
\end{cor}

\begin{cor}\label{L2 boundedness}
For every $m \in \ell_\infty(\mathbb{Z})$ there is a map $\tilde{T}_m^{(2)}: L_2(\GG_q) \rightarrow L_2(\GG_q)$ extending \eqref{Eqn=FourierSchur} by
\[
\tilde{T}_m^{(2)} \circ \k^{(1)}_{\infty, 2} = \k^{(1)}_{\infty, 2} \circ \tilde{T}_m
\]
which is bounded with norm at most $\Vert m \Vert_\infty$. If $m(0) = 0$ then  $\tilde{T}_m^{(2)}: L_2(\GG_q) \rightarrow L_2^\circ(\GG_q)$.
\end{cor}
\begin{proof}
Define the $\varphi$-GNS inner product on $\Pol(\bG_q)$ by $\langle x, y \rangle = \varphi(x^\ast y)$  and denote the associated GNS space by $\mc{H}_\vphi$. By Lemma \ref{Lem=PeterWeyl} and Proposition \ref{Prop=Coefficients} we see that $\tilde{T}_m: \Pol(\bG_q) \rightarrow \Pol(\bG_q)$ is bounded with respect to this inner product with bound at most $\Vert m \Vert_\infty$.  Hence it extends to a map $\tilde{T}_m^\vphi: \mc{H}_\vphi \to \mc{H}_\vphi$.
By \cite[Section 2.2]{Ter82} we have that
\[
\Pol(\bG_q) \rightarrow L_2(\bG_q): x \mapsto x D_\varphi^{1/2}
\]
is an isometry  with respect to this inner product on the left and hence extends to a unitary map   $U: \mc{H}_\vphi \to L_2(\GG_q)$. Then the map $\tilde{T}_m^{(2)} := U \tilde{T}_m^\vphi U^*: L_2(\GG_q) \to L_2(\GG_q)$ satisfies the conditions.  The final statement is Lemma \ref{image in circ}.

\end{proof}






\subsection{Transference principle and construction of $\BMO(\bG_q)$}\label{SubSect=BMODef}

In this subsection we construct the BMO spaces corresponding to $\GG_q = SU_q(2), q \in (-1, 1) \backslash \{ 0 \}$ that we need for the upper endpoint estimate. 
The main tool behind both the construction of the BMO spaces and the proof of the actual upper endpoint estimate is the transference principle of Lemma \ref{transference}. The idea is to obtain properties of Fourier-Schur multipliers on $L_\infty(\GG_q)$ from properties of Fourier multipliers on $L_\infty(\T)$. \\

Recall that $\zeta_i: \T \to \T$ was defined by $z \mapsto z^i$ and let $e_{i,j}$ be the matrix units in $\mc{B}(\ell_2(\N))$. We define the unitary
\[
    U = \sum^\infty_{i=0} e_{i,i} \ot  1_{\mc{B}(\ell_2(\Z))} \ot \zeta_i \in \mc{B}(\mc{H}) \bar{\ot} L_\infty(\T),
\]
and the injective normal $*$-homomorphism
\[
    \pi: \mc{B}(\mc{H}) \to \mc{B}(\mc{H}) \bar{\ot} L_\infty(\T): x \mapsto U^*(x \ot 1) U.
\]

\begin{lem}
We have for $k \in \mathbb{Z}, l,m \in \mathbb{N}$ that
\begin{equation} \label{formula pi}
    \pi (\a^k \g^l (\g^*)^m ) =\a^k \g^l (\g^*)^m \ot \zeta_k.
\end{equation}
\end{lem}
\begin{proof}
For $\xi \in L_2(\T)$, $i \in \N$, $j \in \Z$,
\begin{align*}
 &   \pi(\a^k \g^l (\g^*)^m)  (e_i \ot f_j \ot \xi) \\
 =& U^* ( \a^k \g^l (\g^*)^m \ot \id) (e_i \ot f_j \ot \zeta_i \xi) \\
    = &U^* \sqrt{(1-q^{2i})(1 - q^{2i-2})\dots(1 - q^{2i-2k+2})} q^{i(l+m)} e_{i-k} \ot f_{j + l - m} \ot \zeta_i \xi \\
    =&\sqrt{(1-q^{2i})(1 - q^{2i-2})\dots(1 - q^{2i-2k+2})} q^{i(l+m)} e_{i-k} \ot f_{j + l - m} \ot \zeta_k \xi \\
    =& (\a^k \g^l (\g^*)^m \ot \zeta_k) (e_i \ot f_j \ot \xi).
\end{align*}
\end{proof}

This implies that $\pi$ maps $\rm{Pol}(\GG_q)$ into $\rm{Pol}(\GG_q) \ot L_\infty(\T)$. Hence by density, it maps $L_\infty(\GG_q)$ into $L_\infty(\GG_q) \bar{\ot} L_\infty(\T)$.
We denote $\iota_{\mc{M}}$ for the identity operator $\mc{M} \to \mc{M}$ on a von Neumann algebra $\mc{M}$, reserving $1_{\mc{M}}$ for the unit of $\mc{M}$. The following identity is now immediate.
We refer to this identity as the `transference principle'.

\begin{lem} \label{transference}
Let $\tilde{m} \in \ell_\infty(\Z)$. For $k  \in \Z, l,m \in \mathbb{N}$ we have
\begin{equation*}  (\iota_{L_\infty(\GG_q)} \ot T_{\tilde{m}}) \pi(\a^k \g^l (\g^*)^m) = \tilde{m}(k) \pi(\a^k \g^l (\g^*)^m). \end{equation*}
\end{lem}

Set again the Heat multipliers $h_t(k) = e^{-tk^2}, k \in \Z, t \geq 0$. Let us define a semigroup on $L_\infty(\GG_q) \bar{\ot} L_\infty(\T)$ by  $S = (S_t)_{t \geq 0}$ with  $S_t := \iota_{L_\infty(\GG_q)} \ot T_{h_t}$. Recall that $(T_{h_t})_{t \geq 0}$ is a Markov semigroup (see Section \ref{Sect=torus}). By approximation with elements from the algebraic tensor product and the text following Proposition \ref{wk* cont extension}, one can prove that $S$ is also a Markov semigroup. From this and the transference principle, we can now prove that the semigroup $(\Phi_t)_{t \geq 0}$ we defined in Section \ref{SubSect=FourierSchur} is actually a well-defined Markov semigroup.

\begin{prop} \label{prop 2.3}
The family of maps given by the assignment
\[
\Phi_t(\a^k \g^l (\g^*)^m ) = e^{-tk^2} \a^k \g^l (\g^*)^m, \qquad k \in \mathbb{Z}, l,m \in \mathbb{N}, t \geq 0,
\]
extends to a  Markov semigroup of Fourier-Schur multipliers $\Phi:= (\Phi_t)_{t \geq 0}$ on $L_\infty(\GG_q)$  satisfying
\[
\pi \circ \Phi_t = S_t \circ \pi.
\]
 Moreover, the semi-group is modular.
\end{prop}
\begin{proof}
By Lemma \ref{transference} we have the commutative diagram:
\begin{center}
\begin{tikzcd}
L_\infty(\GG_q) \bar{\ot} L_\infty(\T) \arrow{r}{S_t} & L_\infty(\GG_q) \bar{\ot} L_\infty(\T)   \\
\rm{Pol}(\GG_q) \arrow{u}{\pi} \arrow{r}{\Phi_t} & L_\infty(\GG_q) \arrow{u}{\pi}
\end{tikzcd}
\end{center}
$\pi$ is a normal injective $\ast$-homomorphism so that we may view $L_\infty(\GG_q)$ as a (unital) von Neumann subalgebra of $L_\infty(\GG_q) \bar{\ot} L_\infty(\T)$. We find that $\Phi_t$, being the restriction of $S_t$ to $\Pol(\GG_q)$,  is also a normal ucp map. This means that $\Phi_t$ extends to a normal ucp map on $L_\infty(\GG_q)$.  By the same argument, we deduce strong continuity of $t \mapsto \Phi_t(x)$. This shows properties (i) and (iii) of Definition \ref{Dfn=Markov}.

To show property (ii), we recall (see \eqref{haar state}) that the Haar functional $\vphi$ on $\GG_q$ is non-zero on basis elements $\a^k \g^l (\g^*)^m$ only if $k=0, l=m$.  If $x = \a^k \g^l (\g^*)^m$, $y = \a^{k'} \g^{l'} (\g^*)^{m'}$, then $xy = C \a^{k + k'} \g^{l + l'} (\g^*)^{m + m'}$ for some constant $C$. This shows that $\vphi(x \Phi_t(y)) = \vphi(\Phi_t(x) y)$ on basis elements $x, y$, and hence everywhere.

Finally, by the formula for the modular automorphism group \eqref{sigma_t}, we find that $\Phi_t$ is $\vphi$-modular.
\end{proof}

\bigskip

We define corresponding BMO spaces for this semigroup. We use the shorthand notation $\BMO(\GG_q)$ for $\BMO(L_\infty(\GG_q), \Phi)$, and similarly for the column and row spaces.   We can also define a BMO-norm $\| \cdot \|_{\BMO_S}$ on $(L_\infty(\GG_q) \bar{\ot} L_\infty(\T))^\circ$.  We will do some of the estimates within the normed spaces $(L_\infty^\circ(\GG_q), \| \cdot \|_{\BMO_\Phi})$   and $((L_\infty(\GG_q) \bar{\ot} L_\infty(\T))^\circ, \| \cdot \|_{\BMO_S})$ to avoid some technicalities. \\




\begin{lem} \label{pi extend to BMO}
The map $\pi$ is isometric as a map between normed spaces
\[
  \pi: (L^\circ_\infty(\GG_q), \| \cdot \|_{\BMO_\Phi}) \to ((L_\infty(\GG_q) \bar{\ot} L_\infty(\T))^\circ, \| \cdot \|_{\BMO_S}).
\]
\end{lem}

\begin{proof}
This follows from the commutative diagram of Proposition \ref{prop 2.3} and the fact that $\pi$ is an injective, hence isometric, $*$-homomorphism $L_\infty(\GG_q) \to L_\infty(\GG_q) \bar{\ot} L_\infty(\T)$. Indeed, for $x \in L_\infty(\GG_q)^\circ$, we have that
\[
   \|S_t(\pi(x))\|_\infty = \|(\pi \circ \Phi_t)(x)\|_\infty \to 0,
\]
which implies in particular $\s$-weak convergence. Hence $\pi(x) \in (L_\infty(\GG_q) \bar{\ot} L_\infty(\T))^\circ$. Also,
\begin{align*}
\|\pi(x)\|_{\rm{BMO}^c_S}^2 &= \sup_{t \geq 0} \|S_t ( |\pi(x) - S_t (\pi(x))|^2) \| = \sup_{t \geq 0} \|S_t ( |\pi(x) - \pi(\Phi_t(x))|^2) \| \\
&= \sup_{t \geq 0} \|S_t (\pi( |x - \Phi_t(x)|^2)) \| = \sup_{t \geq 0} \|\pi( \Phi_t(|x - \Phi_t(x)|^2)) \| \\
&= \sup_{t \geq 0} \|  \Phi_t(|x - \Phi_t(x)|^2) \|  = \|x\|_{\rm{BMO}^c_\Phi}^2.
\end{align*}
Replacing $x$ by $x^*$ yields isometry for the row BMO-norm from which it follows that $\pi$ is isometric on BMO as well.
\end{proof}

\bigskip

\subsection{$L_\infty$-$\BMO$ estimate}

We proceed to prove an upper end point estimate for $\tilde{T}_m$. Recall that we defined a ``weak-$\ast$ topology" on $\BMO(\GG_q)$ in Remark \ref{Rem=Weak*TopologyForBMO}.

\begin{thm} \label{T_m prop}
Let $m \in \ell_\infty(\Z)$ with $m(0) = 0$ be such that $T_m: L_\infty(\T) \to \BMO(\T)$ is completely bounded. Then there exists a bounded weak-$\ast$/weak-$\ast$ continuous map
\[
    \tilde{T}^{(\infty)}_m: L_\infty(\GG_q) \to \rm{BMO}(\GG_q),
\]
satisfying $\tilde{T}_m^{(\infty)}(x) = \k_{\infty, 1}^{(0)}(\tilde{T}_m(x))$ for $x \in \rm{Pol}(\GG_q)$. Moreover,
\begin{equation}\label{Eqn=CBEstimate}
\Vert    \tilde{T}^{(\infty)}_m: L_\infty(\GG_q) \to \BMO(\GG_q) \Vert \leq  \Vert T_m: L_\infty(\T) \to \BMO(\T) \Vert_{cb}.
\end{equation}
\end{thm}

  The proof consists of the following two lemmas. We first prove a BMO-norm estimate of $\tilde{T}_m$ for the polynomial algebra, using again the transference principle from Lemma \ref{transference}.

\begin{lem} \label{Lem=BMO estimate}
Let $m \in \ell_\infty(\Z)$ with $m(0) = 0$ be such that $T_m: L_\infty(\T) \to \BMO(\T)$ is completely bounded. Then for $x \in \rm{Pol}(\GG_q)$:
\begin{equation} \label{Eqn=TmEstimate}
    \|\tilde{T}_m(x)\|_{\BMO_\Phi} \leq \Vert T_m: L_\infty(\T) \to \BMO(\T) \Vert_{cb} \|x\|_\infty.
\end{equation}

\end{lem}

\begin{proof}
 By Lemma \ref{image in circ}, $\tilde{T}_m$ maps $\rm{Pol}(\GG_q)$ to $L^\circ_\infty(\GG_q)$.
Note that $\pi$ sends $\rm{Pol}(\GG_q)$ to $L_\infty(\GG_q) \ot \rm{Pol}(\T)$ and $\iota_{L_\infty(\GG_q)} \ot T_m$ sends $L_\infty(\GG_q) \ot \rm{Pol}(\T)$ to $L_\infty(\GG_q) \ot L_\infty^\circ(\T) \subseteq (L_\infty(\GG_q) \bar{\ot} L_\infty(\T))^\circ$ (see also Appendix \ref{bmo operator space}). Now Lemma \ref{transference} gives us a commutative diagram like in Proposition \ref{prop 2.3}.
\begin{center}
\begin{tikzcd}
L_\infty(\GG_q) \ot \rm{Pol}(\T) \arrow{rr}{\iota_{L_\infty(\GG_q)} \ot T_m} && (L_\infty(\GG_q) \bar{\ot} L_\infty(\T))^\circ   \\
\rm{Pol}(\GG_q) \arrow{u}{\pi} \arrow{rr}{\tilde{T}_m} && L_\infty^\circ(\GG_q) \arrow{u}{\pi}
\end{tikzcd}
\end{center}
  Note that in particular the restriction $T_m: \rm{Pol(\T)} \to (L_\infty^\circ(\T), \|\cdot \|_{\BMO})$ is completely bounded.  Now Lemma \ref{pi extend to BMO} and Proposition \ref{cb remark} allows us to find a BMO-estimate on $\tilde{T}_m$ for $x \in \rm{Pol}(\GG_q)$:

\begin{equation*}
\begin{split}
 \| \tilde{T}_m(x) \|_{\BMO_\Phi} &= \|\pi \circ \tilde{T}_m(x)\|_{\BMO_S} = \|(\iota_{L_\infty(\GG_q)} \ot T_m) \circ \pi (x) \|_{\BMO_S} \\
& \leq \|T_m\|_{cb} \|\pi(x)\| = \|T_m\|_{cb} \|x\|_{\infty}.
\end{split}
\end{equation*}
where $\|T_m\|_{cb} = \Vert T_m: L_\infty(\T) \to \BMO(\T) \Vert_{cb}$.
\end{proof}

Recall that $\k^{(0)}_{\infty, 1}$ isometrically embeds the normed space $(L_\infty^\circ(\GG_q), \| \cdot \|_{\BMO_\Phi})$ into $\BMO(\GG_q)$. Now define $\tilde{T}_m^{(\infty)} = \k_{\infty, 1}^{(0)} \circ \tilde{T}_m$, which we may consider as a bounded map from $\rm{Pol}(\GG_q)$ to $\BMO(\GG_q)$ by Lemma \ref{Lem=BMO estimate}. It remains to prove that this map extends to $L_\infty(\GG_q)$. The proof is essentially that of \cite[Lemma 1.6]{JMP14} together with a number of technicalities that we overcome here.

\begin{lem} \label{Lem=weak-* BMO extension}
$\tilde{T}_m^{(\infty)}$ has a weak-$\ast$/weak-$\ast$ continuous extension to $L_\infty(\GG_q) \to \BMO(\GG_q)$.
\end{lem}

\begin{proof}
Let $h_1^c(\GG_q) := h_1^c(L_\infty(\GG_q), \Phi)$ and $h_1^r(\GG_q) := h_1^r(L_\infty(\GG_q), \Phi)$ be the preduals constructed in Section \ref{bmo-h1 duality}. We will construct maps $S_c: h_1^c(\GG_q) \to L_1(\GG_q)$ and $S_r: h_1^r(\GG_q) \to L_1(\GG_q)$ such that their adjoints are equal and extend $\tilde{T}_m^{(\infty)}$.

\vspace{0.3cm}

\noindent \textit{Construction of maps $S_c$ and $S_r$.} We construct the map $S_c: h^c_1(\GG_q) \to L_1(\GG_q)$ by proving that the map $\k_{2,1}^{(1)} \circ (\tilde{T}_m^{(2)})^*$ is bounded as a map $L_2^\circ(\GG_q) \to L_1(\GG_q)$ with respect to $\| \cdot \|_{h^c_1(\GG_q)}$ on the left. For $y \in L_2^\circ(\GG_q)$ and $z \in \rm{Pol}(\GG_q)$ we find
\begin{equation}\label{Eqn=PairingId}
\la z,  (\tilde{T}_m^{(2)})^*(y) D_\vphi^{1/2} \ra  =  \la  D_\vphi^{1/2} z, (\tilde{T}_m^{(2)})^*(y)  \ra  =  \la D_\vphi^{1/2} \tilde{T}_m(z), y \ra.
\end{equation}
By the Kaplansky density theorem  and \cite[Theorem II.2.6]{Takesaki1} the unit ball of $\rm{Pol}(\GG_q)$ is weak-$\ast$ dense in the unit ball of $L_\infty(\GG_q)$. Hence for $y \in L_2^\circ(\GG_q)$ we find:
\[
\begin{split}
 \|\k^{(1)}_{2, 1}((\tilde{T}_m^{(2)})^*y) \|_{L_1(\GG_q)} = & \sup_{z \in \rm{Pol}(\GG_q)_{\leq 1}} |\la z,  (\tilde{T}_m^{(2)})^*(y) D_\vphi^{1/2} \ra| \\
  =  &\sup_{z \in \rm{Pol}(\GG_q)_{\leq 1}} |\la D_\vphi^{1/2} \tilde{T}_m(z), y \ra| \leq \|T_m\|_{cb} \|y\|_{h^c_1(\GG_q)}.
\end{split}
\]
In the last step we used that $\|\k_{\infty, 2}^{(-1)}(\tilde{T}_m (z))\|_{\BMO^r} = \|\tilde{T}_m (z)\|_{\BMO^r} \leq \|T_m\|_{cb} \|z\|_\infty$. We conclude that $\k_{2,1}^{(1)} \circ (\tilde{T}_m^{(2)})^*$ extends to a bounded map
\[
S_c: h^c_1(\GG_q) \to L_1(\GG_q).
\]
In a similar manner we can prove that the map $\k_{2,1}^{(-1)} \circ (\tilde{T}_m^{(2)})^*$ extends to a bounded map
\[
S_r: h_1^r(\GG_q) \to L_1(\GG_q).
\]
By taking limits in \eqref{Eqn=PairingId}, we can prove the following equalities for $z \in \rm{Pol}(\GG_q)$, $y_a \in h^c_1(\GG_q)$ and $y_b \in h^r_1(\GG_q)$:
\begin{equation} \label{Sc Sr calc}
    \la z, S_c(y_a) \ra = \la D_\vphi^{1/2} \tilde{T}_m(z) , y_a \ra, \ \ \ \ \ \ \ \
    \la z, S_r(y_b) \ra = \la  \tilde{T}_m(z) D_\vphi^{1/2}, y_b \ra
\end{equation}

\vspace{0.3cm}

\noindent \textit{Analysis of the adjoint maps.} Now consider the adjoint maps  $S_r^*: L_\infty(\GG_q) \to \BMO^c(\GG_q)$ and $S_c^*: L_\infty(\GG_q) \to \BMO^r(\GG_q)$. They are weak-$\ast$/weak-$\ast$ continuous by Proposition \ref{extension property}. By composition, we get maps
\[
    T_c:= \k_{2,1}^{(-1)} \circ S_r^*: L_\infty(\GG_q) \to L_1^\circ(\GG_q), \qquad T_r := \k_{2,1}^{(1)} \circ S_c^*: L_\infty(\GG_q) \to L_1^\circ(\GG_q).
\]
Let $z \in \Pol(\GG_q)$ and $y \in h_1^c(\GG_q)$. Then \eqref{Sc Sr calc} yields
\[
    \la S_c^*(z), y \ra = \la z, S_c(y) \ra = \la D_\vphi^{1/2} \tilde{T}_m(z), y \ra
\]
so $S_c^*$ extends $\k_{\infty,2}^{(-1)} \circ \tilde{T}_m$. Hence, $T_r$ is a weak-$\ast$/weak-$\ast$ continuous extension of $\k_{\infty,1}^{(0)} \circ \tilde{T}_m = \tilde{T}_m^{(\infty)}$. In a similar way, we find that $T_c$ is a weak-$\ast$/weak-$\ast$ continuous extension of $\tilde{T}_m^{(\infty)}$. In particular, $T_c$ and $T_r$ coincide on $\Pol(\GG_q)$. It remains to prove that $T_c = T_r$; this implies that the image of this map is contained in $\BMO(\GG_q)$, and hence it is the weak-$\ast$/weak-$\ast$ continuous extension of $\tilde{T}_m^{(\infty)}$ that we were looking for. \\

\noindent \emph{Proof of the equality $T_r = T_c$.} Let $x \in L_\infty(\GG_q)$ and take a net $x_\l \in \Pol(\GG_q)$ such that $x_\l \to x$ in the weak-$\ast$ topology. By the weak-$\ast$/weak-$\ast$ continuity of $S_r^*$, we have $S_r^*(x_\l) \to S_r^*(x) =: y_c \in \BMO^c(\GG_q)$ in the weak-$\ast$ topology. Similarly, $S_c^*(x_\l) \to S_c^*(x) =: y_r \in \BMO^r(\GG_q)$ in the weak-$\ast$ topology. We need to prove that $D_\vphi^{1/2} y_c = y_r D_\vphi^{1/2}$ in $L_1^\circ(\GG_q)$. \\

First let $z \in L_\infty^\circ(\GG_q)$. In that case, we have
\[
    \la S_r^*(x_\l), z D_\vphi^{1/2} \ra_{\BMO^c, h_1^r} = \la T_c(x_\l), z \ra_{1,\infty} = \la T_r(x_\l), z \ra_{1 ,\infty} = \la S_c^*(x_\l), D_\vphi^{1/2} z\ra_{\BMO^r, h_1^c}.
\]
Hence, 
\[
\begin{split}
    \la D_\vphi^{1/2} y_c, z \ra_{1,\infty} &= \la y_c, z D_\vphi^{1/2} \ra_{\BMO^c, h_1^r} = \lim_{\l} \la S_r^*(x_\l), z D_\vphi^{1/2} \ra_{\BMO^c, h_1^r} \\
    &= \lim_{\l} \la S_c^*(x_\l), D_\vphi^{1/2}z \ra_{\BMO^r,h_1^c} = \la y_r, D_\vphi^{1/2}z \ra_{\BMO^r, h_1^c} = \la y_r D_\vphi^{1/2}, z \ra_{1,\infty}.
\end{split}
\]
Now let $z \in L_\infty(\GG_q)$. Let $\E$ be the projection of $L_\infty(\GG_q)$ onto $L_\infty^\circ(\GG_q)$. Then as in Proposition \ref{prop extension to Lp}, we get a projection $\E^{(1)}: L_1(\GG_q) \to L_1^\circ(\GG_q)$ which is the adjoint of $\E$. Hence,
\[
    \la D_\vphi^{1/2} y_c, z \ra = \la \E^{(1)}(D_\vphi^{1/2} y_c), z \ra = \la D_\vphi^{1/2} y_c, \E(z) \ra = \la y_r D_\vphi^{1/2}, \E(z) \ra = \la y_r D_\vphi^{1/2}, z \ra. 
\]
We conclude that $D_\vphi^{1/2} y_c = y_r D_\vphi^{1/2}$; this finishes the proof. 
\end{proof}

\begin{proof}[Proof of Theorem \ref{T_m prop}]
The existence of $\tilde{T}_m^{(\infty)}$ follows from Lemma \ref{Lem=BMO estimate} and \ref{Lem=weak-* BMO extension}. The inequality in \eqref{Eqn=CBEstimate}
follows from \eqref{Eqn=TmEstimate} and the Kaplansky density theorem.
\end{proof}

\subsection{Consequences for $L_p$-Fourier Schur multipliers}

\begin{thm}\label{Thm=FourierSchurLpVersion}
Let $m \in \ell_\infty(\Z)$ with $m(0) = 0$ be such that the Fourier multiplier $T_m: L_\infty(\T) \to \BMO(\T)$ is completely bounded. Let $\tilde{T}_m: \rm{Pol}(\GG_q) \to \rm{Pol}(\GG_q)$ be the Fourier-Schur multiplier with symbol $(m(-i-j))_{i,j,l}$ with respect to the basis described in \eqref{basis vectors}, where $l \in \frac{1}{2} \mathbb{N}$ indexes the corepresentation and $1 \leq i,j \leq 2l +1$. Then for $1 \leq p < \infty$, $\tilde{T}_m$ extends to a bounded map
\[
    \tilde{T}_m^{(p)}: L_p(\GG_q) \to L_p^\circ(\GG_q),
\]
where by `extension' we mean that $\tilde{T}_m^{(p)}(\k_{\infty, p}^{(1)}(x)) = \k_{\infty, p}^{(1)}(\tilde{T}_m(x))$

\end{thm}

\begin{proof}
Proposition \ref{L2 boundedness} and Theorem \ref{T_m prop} show that $\tilde{T}_m^{(\infty)}$ and $\tilde{T}_m^{(2)}$ together are compatible morphisms. Therefore, by Riesz-Torin (see e.g. Theorem 2.5 from \cite{Cas11} and the rest of that paragraph), we get bounded maps on the interpolation spaces. Since $\Phi$ admits a Markov dilation (see Proposition \ref{prop 2.5}), Theorem \ref{bmo interpolation thm} tells us that
\[
    [\BMO(\GG_q), L_2^\circ(\GG_q)]_{2/p} \approx L^\circ_p(\GG_q).
\]
Also we have by \cite{Kos84} that
\[
    [L_\infty(\GG_q), L_2(\GG_q)]_{2/p} \approx L_p(\GG_q).
\]
This proves that for $2 \leq p < \infty$ we can construct bounded maps $\tilde{T}^{(p)}_m: L_p(\GG_q) \to L_p^\circ(\GG_q)$ that extend $\tilde{T}_m$ - or more precisely, they satisfy $\tilde{T}_m^{(p)}(\k_{\infty, p}^{(1)}(x)) = \k_{\infty, p}^{(1)}(\tilde{T}_m(x))$ for all $x \in \rm{Pol}(\GG_q)$.

Now consider $1 \leq p < 2$ and let $p'$ be such that $\frac1p + \frac1{p'} = 1$. Then the adjoint map $\tilde{T}_m^*$ is simply the Fourier multiplier with symbol $\bar{m}$, and hence by the above argument $\tilde{T}_m^*$ extends to a map on $L_{p'}(\GG_q)$. Hence the map $\tilde{T}_m^{(p)}: L_p(\GG_q) \to L_p(\GG_q)$ given by the double adjoint is the extension we were looking for.
\end{proof}

\begin{rem}
The condition that $m(0) = 0$ is not very important: if we `add a constant sequence to $m$', i.e. we switch to the map $T_{m + \l 1} = T_m + \l \iota_{L_\infty(\T)}$, then this map still `extends' (in the sense of the theorem) to a bounded map $L_p(\GG_q) \to L_p(\GG_q)$. 
\end{rem}

\begin{rem}
In \cite[Lemma 3.3]{JMP14} classes of completely bounded multipliers $L_\infty(\mathbb{T}) \rightarrow \BMO(\mathbb{T})$ have been constructed. Further, in \cite[Lemma 1.3]{JMP14} the connection between classical BMO-spaces and non-commutative semi-group BMO spaces is established giving further examples. This shows that indeed the class of symbols $m$ to which Theorem \ref{Thm=FourierSchurLpVersion} applies is non-empty and contains a reasonable class of examples.
\end{rem}

%

%

\appendix

\section{ Completely bounded maps with respect to the BMO-norm} \label{bmo operator space}

Throughout this section, let $\mc{M} \subseteq \mc{B}(\mc{K})$ be a $\s$-finite von Neumann algebra with n.f. state $\vphi$ and Markov semigroup $\Phi = (\Phi_t)_{t \geq 0}$. Fix some $n \geq 2$. Then the maps $\iota_{M_n} \ot \Phi_t$ define a Markov semigroup on $M_n(\mc{M})$. Hence we can define the matrix BMO-norms $\| \cdot \|_{\BMO_n}$ on $M_n(\mc{M})^\circ$ with respect to the semigroup $S_n := (\iota_{M_n} \ot \Phi_t)_{t \geq 0}$.
Through a straightforward calculation, one also checks that $M_n(\mc{M})^\circ = M_n(\mc{M}^\circ)$. Hence the above norms define matrix norms on $\mc{M}^\circ$. It is not hard to prove that these norms turn $\mc{M}^\circ$ into an operator space, which we denote by $(\mc{M}^\circ, \| \cdot \|_{\BMO})$. We leave the details to the reader.

Let $\mc{N} \subseteq \mc{B}(\mc{H})$ be a $\s$-finite von Neumann algebra. Then $\mc{N} \bar{\ot} \mc{M}$ is again a $\s$-finite von Neumann algebra. Similarly as in the matrix case, $S := (\iota_{\mc{N}} \ot \Phi_t)_{t \geq 0}$ is a semigroup on $\mc{N} \bar{\ot} \mc{M}$. In line with the main text, we denote $\| \cdot \|_{\BMO_S}$ for the corresponding BMO-norm on $(\mc{N} \bar{\ot} \mc{M})^\circ$. \\

Using the fact that $\mc{N}_* \ot \mc{M}_*$ is dense in $(\mc{N} \bar{\ot} \mc{M})_*$ (see \cite[Chapter 1.22]{Sakai}) one can show that $\mc{N} \ot \mc{M}^\circ \subseteq (\mc{N} \bar{\ot} \mc{M})^\circ$ .\\


\begin{prop} \label{cb remark}
Let $\mc{A} \subseteq \mc{M}$ be a linear subspace and  $T: \mc{A} \to (\mc{M}^\circ, \| \cdot \|_{\BMO})$ be completely bounded. For $x \in \mc{N} \ot \mc{A}$,
\[ \|(\iota_{\mc{N}} \ot T)(x)\|_{\BMO_S} \leq \|T\|_{cb} \|x\|_{\mc{B}(\mc{H} \ot_2 \mc{K})}. \]
\end{prop}

\begin{proof}
Take $x \in \mc{N} \ot \mc{A}$ and write  $x = \sum_n x_n \ot x_n'$.
Let $z = (\iota_{\mc{N}} \ot T)(x) \in \mc{N} \ot \mc{M}^\circ$.   Setting $w_n = T(x_n')$ we have $z = \sum_n x_n \ot w_n$.
For  a finite dimensional subspace $F \subseteq \mc{H}$ let $P_F$ be the projection onto $F$. Denote $x_n^F = P_Fx_nP_F$ the truncation of $x_n$ to $F$. Denote $z^F = \sum_n x_n^F \ot w_n$ and  $x^F = \sum_n x_n^F \ot x_n'$.

Now we prove the column estimate.  Let $\xi \in \mc{H} \ot \mc{K}$ (algebraic tensor product) and write $\xi = \sum_k \xi_k \ot \eta_k$. Define $F \subseteq \mc{H}$ to be
\[ F = \Span\{ \xi_k, x_m\xi_k, x_n^*x_m\xi_k\ |\ n,m,k\}. \]
Then we note that $F$ is finite dimensional and $(x_n^F)^* x_m^F \xi_k = x_n^*x_m \xi_k$,
Let $t \geq 0$ be arbitrary. Writing out the expression in the column BMO-norm gives
\[
(\iota_{\mc{N}} \ot \Phi_t)(| z - (\iota_{\mc{N}} \ot \Phi_t)(z)|^2) = \sum_{n, m} x_n^* x_m \ot \Phi_t((w_n - \Phi_t(w_n))^* (w_m - \Phi_t(w_m))).
\]
Hence, denoting $S_F := (\iota_{\mc{B}(F)} \ot \Phi_t)_{t \geq 0}$,

\begin{eqnarray*}
&&\| (\iota_{\mc{N}} \ot \Phi_t) ( |z - (\iota_{\mc{N}} \ot \Phi_t)(z) |^2 ) \xi \|_{\mc{H} \ot_2 \mc{K}} \\
&=& \| (\iota_{\mc{B}(\mc{F})} \ot \Phi_t)  ( |z^F - (\iota_{\mc{B}(\mc{F})} \ot \Phi_t)(z^F) |^2 ) \xi \|_{F \ot \mc{K}} \\
&\leq & \| (\iota_{\mc{B}(\mc{F})} \ot \Phi_t) ( |z^F - (\iota_{\mc{B}(\mc{F})} \ot \Phi_t)(z^F) |^2 )\|_{\mc{B}(F \ot \mc{K})} \|\xi\| \\
&\leq & \|z^F \|^2_{\BMO^c_{S_F}} \|\xi\| = \|(\iota_{\mc{B}(F)} \ot T)(x^F)\|^2_{\BMO^c_{S_F}} \|\xi\| \leq \|T\|^2_{cb} \|x\|^2_{\mc{B}(\mc{H} \ot_2 \mc{K})} \|\xi\|.
\end{eqnarray*}
 In the last step, we used that $T$ is also completely bounded when considering $\| \cdot \|_{\BMO^c}$ on the right.  Taking the supremum over all $\xi \in \mc{H} \otimes \mc{K}$ with  $\|\xi\| = 1$ and $t \geq 0$, we conclude
\[ \|(\iota_{\mc{N}} \ot T)(x)\|_{ \BMO^c_S } \leq \|T\|_{cb} \|x\|_{\mc{B}( \mc{H}  \ot_2 \mc{K})} \]
The row BMO estimate follows similarly, from which the BMO estimate follows.
\end{proof}

\begin{rem}
In the case where $\mc{M}$ is a finite von Neumann algebra, we can extend the operator space structure to $\BMO(\mc{M}, \Phi)$. In the $\s$-finite case however, it seems to be more difficult than expected to prove that $M_n(\BMO(\mc{M}, \Phi)) \subseteq \BMO(M_n(\mc{M}), \iota_{M_n} \ot \Phi)$.
\end{rem}

\bigskip

\section{A Markov dilation of the Markov semigroup $\Phi$} \label{markov dilation section}

\begin{defi}
We say that a Markov semigroup $\Phi$ on  a $\s$-finite von Neumann algebra $\mc{M}$ with faithful normal state $\varphi$ admits a \emph{standard Markov dilation} if there exist:
\begin{enumerate}[(i)] \nosep
\item a $\s$-finite von Neumann algebra $\mc{N}$ with normal faithful state $\vphi_\mc{N}$,
\item an increasing filtration $(\mc{N}_s)_{s \geq 0}$ with $\vphi_\mc{N}$-preserving conditional expectations $\mc{E}_s: \mc{N} \to \mc{N}_s$,
\item a $*$-homomorphisms $\pi_s: \mc{M} \to \mc{N}_s$  such that $\varphi_\mc{N} \circ \pi_s = \varphi_\mc{N}$ and
\[
    \mc{E}_s(\pi_t(x)) = \pi_s(\Phi_{t-s}(x)), \ \ \ \ s < t, x \in \mc{M}.
\]
\end{enumerate}
A Markov dilation is called $\vphi$-modular if it additionally satisfies
\[
    \pi_s \circ \s^\vphi_t = \s^{\vphi_\mc{N}}_t \circ \pi_s, \ \ \ \ s \geq 0, t \in \R.
\]
\end{defi}

One can analogously define the notion of a reversed Markov dilation; we refer to \cite[Definition 5.1]{CJSZ19} for the precise statement. \\

In this subsection, we construct a Markov dilation for the semigroup $\Phi = (\Phi_t)_{t \geq 0}$ on $L_\infty(\GG_q)$ given by
\[
\Phi_t(\a^k \g^l (\g^*)^m) = e^{-tk^2} \a^k \g^l (\g^*)^m, \qquad k \in \mathbb{Z}, l,m \in \mathbb{N},
\]
 as used in Section \ref{section SUq(2)}. \\

To construct the Markov dilation, we use the fact that $L_\infty(\GG_q)$ can be written as the tensor product of two relatively simple von Neumann algebras. This is a well-known fact; we give a sketch of the proof for the convenience of the reader. We let $\mc{L}(\mathbb{Z})$ be the group von Neumann algebra of $\mathbb{Z}$ generated by the left regular representation $\lambda$.

\begin{prop} \label{L(G)}
$L_\infty(\GG_q) = \mc{B}(\ell_2(\N)) \bar{\ot} \mc{L}(\Z). $
\end{prop}

\begin{proof}
Let $T_m$, $T_{\tilde{m}}$ be the multiplication maps on $\ell_2(\N)$ with symbols $m(k) = q^k$, $\tilde{m}(k) = \sqrt{1 - q^{2k}}$. Then we can write
\[
    \g = T_m \ot \l_{1, \Z}, \ \ \ \ \a = (\l_{1,\N}^* T_{\tilde{m}}) \ot 1
\]
where we denote $\l_{1,\Z}$ and $\l_{1,\N}$ for the right shift on $\ell_2(\Z)$ and $\ell_2(\N)$ respectively. From these expressions it is immediately clear that $L_\infty(\GG_q) \subseteq \mc{B}(\ell_2(\N)) \bar{\ot} \mc{L}(\Z)$. For the other inclusion, note that the partial isometries in the polar decompositions of $\a, \g$ are $1 \ot \l_{1, \Z}$ and $\l_{1, \N}^* \ot 1$ respectively. These elements generate $1 \ot \mc{L}(\Z)$ and $\mc{B}(\ell_2(\N)) \ot 1$ respectively as von Neumann algebras. Hence the other inclusion follows from the definition of the von Neumann algebraic tensor product.
\end{proof}

Through this expression for $L_\infty(\GG_q)$ we will show that $\Phi_t$ can be written as a Schur multiplier. We will need the fact that Schur multipliers are normal.

\begin{prop} \label{Schur mult weak-* cont}
Set $\mc{H} = \ell_2(I)$ for some index set $I$ and let $T: \mc{B}(\mc{H} ) \to \mc{B}(\mc{H})$ be a Schur multiplier with symbol $t = (t_{i,j})_{i,j}$, i.e. $T(e_{i,j}) = t_{i,j} e_{i,j}$. Then $T$ is normal.
\end{prop}

\begin{proof}
Denote $L_1(\mc{H})$ to be the trace class operators and denote $t^T$ to be the transpose of $t$. We claim that $T^\ast|_{L_1(\mc{H})}$ is nothing but the Schur multiplier with symbol $t^T$. Indeed, if $x \in \mc{B}(\mc{H})$, $y \in L_1(\mc{H})$ and $i \in I$ is fixed, then
\[ \la e_i, T(x)y e_i \ra = \sum_{k \in I} t_{i,k} x_{i,k} y_{k,i} = \sum_{k \in I} x_{i,k} t^T_{k,i} y_{k, i} = \la e_i, x (t^T_{i,j} y_{i,j})_{i,j} e_i \ra. \]
Hence
\[ \Tr(T(x) y) = \sum_{i \in I} \la e_i, T(x)y e_i \ra = \sum_{i \in I} \la x e_i, (t^T_{i,j} y_{i,j})_{i,j} e_i \ra = \Tr(x (t^T_{i,j}y_{i,j})_{i,j}). \]
This shows the claim.

Let $y \in L_1(\mc{H})$. By the above $T^\ast(y)$ is an operator on $\mc{H}$ so that we can define its trace-class norm. Then by Hahn-Banach
\[ \|T^*y\|_{L_1(\mc{H})} = \sup_{x \in \mc{B}(\mc{H}): \|x\| \leq 1} |\la x, T^*y \ra| = \sup_{x \in \mc{B}(\mc{H}): \|x\| \leq 1} |\la Tx, y \ra| \leq \|T\| \|y\|_{L_1(\mc{H})}. \]
So $T^*$ restricts to an operator $L_1(\mc{H}) \to L_1(\mc{H})$. Therefore, since $\mc{B}(\mc{H}) = L_1(\mc{H})^\ast$, we see that by Proposition \ref{wk* cont extension} $T = (T^\ast|_{L_1(\mc{H})})^\ast$ is normal.
\end{proof}

\begin{prop} \label{prop 2.5}
The semi-group $\Phi$ admits a (standard and reversed) $\vphi$-modular Markov dilation.
\end{prop}
\begin{proof}

We prove first  that $\Phi_t$ can be written as a Schur multiplier on the left tensor leg of $L_\infty(\GG_q)$. Let $x = \a^k \g^l(\g^*)^m$.   $x$ acts on basis vectors by
\[ e_i \ot f_r \stackrel{x}\mapsto c e_{i-k} \ot f_{r+l-m}, \ \ \ \ \  c:= c_{q, k,l,m,i,r} = \sqrt{(1-q^{2i})(1 - q^{2i-2})\dots(1 - q^{2i-2k+2})} q^{i(l+m)}. \]
In other words, the matrix elements of $x$ are given by
\[
    \la x e_i \ot f_r, e_j \ot f_s \ra = c\ \d_{j, i-k} \d_{s, r+l-m}.
\]

Hence if we define $\Psi_t: \mc{B}(\ell_2(\N)) \to \mc{B}(\ell_2(\N))$ as the Schur multiplier given by $\Psi_t(e_{i,j}) = e^{-t |i-j|^2} e_{i,j}$, then we have
\[
    \Phi_t(x) = e^{-tk^2} x = (\Psi_t \ot \id_{\mc{L}(\Z)})(x)
\]
Hence $\Phi_t$ and $\Psi_t \ot \id_{\mc{L}(\Z)}$ coincide on $\rm{Pol}(\GG_q)$. Since both are normal (Proposition \ref{prop 2.3} for $\Phi_t$ and Proposition \ref{Schur mult weak-* cont} for $\Psi_t$) they must coincide on $L_\infty(\GG_q)$.

The proof from now on is essentially that of \cite{Ricard08} or \cite[Proposition 4.2]{CJSZ19} with the main difference  that the unitary $u$ below only sums over the indices of $\ell_2(\N)$. Let $\e > 0$ be arbitrary. We define a sesquilinear form on the real finite linear span $\mc{H}_0 = \text{Span}_\R\{e_i, i \in \N\} \subseteq \mc{H}$ by setting
\[
    \la \xi, \eta \ra = \sum_{i, j \in \N} e^{-\e(j - i)^2} \xi_i \eta_j, \ \ \ \ \ \xi , \eta \in \mc{H}_0
\]
We define $\mc{H}_\R$ to be the completion of $\mc{H}_0$ with respect to $\la \cdot, \cdot \ra$ after quotienting out the degenerate part. Let $\G = \G(\mc{H}_\R)$ be the associated exterior algebra (see \cite[Section 2.8]{CJSZ19}) with vacuum vector $\O$ and canonical vacuum state $\tau_{\O}$. The dilation von Neumann algebra $(\mc{B}, \vphi_\mc{B})$ will be given by
\[
    \mc{B} = L_\infty(\GG_q) \bar{\ot} \G^{\bar{\ot} \infty}, \qquad \vphi_\mc{B} = \vphi \ot \tau_\O^{\ot \infty}
\]
where the infinite tensor product is taken with respect to $\tau_\O$. Next we describe the dilation homomorphisms $\pi_s$. We consider the unitary
\[
    u = \sum_{i \in \N} e_{i,i} \ot 1_{\mc{L}(\Z)} \ot s(e_i) \ot  1_{\G}^{\ot \infty}  \in L_\infty(\GG_q)  \bar{\ot}  \G^{\bar{\ot} \infty}
\]
which is defined as a strong limit of sums. Let $S: v \mapsto 1 \ot v$ be the tensor shift on $\G^{\bar{\ot} \infty}$, and let $\b: \mc{B} \to \mc{B}$ be defined by $\b(z) = u^* (\iota_{L_\infty(\GG_q)} \ot S)(z) u$. The $*$-homomorphisms $\pi_s: L_\infty(\GG_q) \to \mc{B}$ are given by
\[
    \pi_0 : x \mapsto x \ot 1 \ot 1 \dots, \ \ \ \ \ \pi_k: x \mapsto (\b^k \circ \pi_0)(x), \ \ \ k \geq 1.
\]
One shows by induction that for $x \in L_\infty(\GG_q)$
\[
    \pi_k(x) = \sum_{i, j \in \N} e_{i,i} x e_{j,j} \ot (s(e_i)s(e_j))^{\ot k} \ot 1_{\G}^{\ot \infty}.
\]
By \eqref{haar state} it follows that $\pi_k$ is state-preserving, and by \cite[Proposition XIV.1.11]{Takesaki3}, it is $\vphi$-modular.

Finally, the filtration is given by
\[
    \mc{B}_m =  L_\infty(\GG_q) \bar{\ot} \G^{\ot m} \ot 1_\G^{\ot \infty} \subseteq \mc{B}.
\]
One checks that the associated conditional expectations satisfy
\begin{align*}
&\mc{E}_m(e_{i,i} x e_{j,j} \ot (s(e_i)s(e_j))^{\ot k} \ot \id_{\G}^{\ot \infty}) \\
=& \tau_\O(s(e_i)s(e_j))^{k-m} e_{i,i} x e_{j,j} \ot (s(e_i)s(e_j))^{\ot m} \ot 1_{\G}^{\ot \infty}.
\end{align*}
From this and the identity
\[ \tau_\O(s(e_i)s(e_j)) = \la s(e_j) \O, s(e_i) \O \ra = e^{-\e(j - i)^2} \]
one deduces that indeed
\[
(\mc{E}_m \circ \pi_k)(x) = \pi_m(\Phi_{\e(k-m)}(x)).
\]
So the semigroup $(\Phi_{\e n})_{n \in \N}$ admits a Markov dilation for any $\e>0$. By \cite[Theorem 3.2]{CJSZ19}, $(\Phi_t)_{t \geq 0}$ admits a standard Markov dilation.  This theorem is stated only for finite von Neumann algebras, but it also holds in the $\s$-finite case with the same proof mutatis mutandis. A reversed Markov dilation can be obtained by essentially the same argument and a $\s$-finite analogue of \cite[Theorem 5.3]{CJSZ19}.


\end{proof}

\end{document}